\documentclass[12pt,a4paper,onecolumn]{article}

\title{Nonhomogeneous boundary condition for spectral non-local operators}

\author{Ivan Bio\v{c}i\'{c} and Vanja Wagner}
\date{}

\usepackage[utf8]{inputenc}
\usepackage{fullpage}
\usepackage{amsmath}
\usepackage{amsthm}
\usepackage{amsfonts}
\usepackage{mathtools}
\usepackage[dvipsnames]{xcolor} 
\usepackage{amssymb} 
\usepackage{enumerate,enumitem}
\usepackage{comment,verbatim}
\usepackage{hyperref}
\usepackage{aligned-overset}
\usepackage{yfonts,dsfont}
\usepackage{cancel,xcolor,soul} 

\usepackage{hyperref}
\hypersetup{
	colorlinks=true,
	linkcolor=magenta,
}
\DeclareMathOperator\supp{supp}

\newtheorem{thm}{Theorem}[section]
\newtheorem{prop}[thm]{Proposition}
\newtheorem{cor}[thm]{Corollary}
\newtheorem{defn}[thm]{Definition}
\newtheorem{lem}[thm]{Lemma}
\theoremstyle{definition}
\newtheorem{rem}[thm]{Remark}

\newtheorem*{assumption*}{\assumptionnumber}
\providecommand{\assumptionnumber}{}
\makeatletter
\newenvironment{assumption}[2]
{%
	\renewcommand{\assumptionnumber}{(\textbf{#1#2})}%
	\begin{assumption*}%
		\protected@edef\@currentlabel{(\textbf{#1#2})}%
	}
	{%
	\end{assumption*}
}
\makeatother

\newcommand{\R}{\mathbb{R}}
\newcommand{\N}{\mathbb{N}}
\newcommand{\p}{\mathds{P}}
\newcommand{\ex}{\mathds{E}}
\newcommand{\G}{\mathds{G}}
\newcommand{\subsub}{\subset\subset}
\newcommand{\1}{\mathbf 1}
\newcommand{\de}{\delta_D}
\newcommand{\wh}{\widehat}
\newcommand{\wt}{\widetilde}
\newcommand{\la}{\langle}
\newcommand{\ra}{\rangle}
\newcommand{\ua}[1]{\underline{a}_{#1}}
\newcommand{\oa}[1]{\overline{a}_{#1}}
\newcommand{\ud}[1]{\underline{\delta}_{#1}}
\newcommand{\od}[1]{\overline{\delta}_{#1}}

\newcommand{\LLL}{L^1(D,V(\de(x))dx)}
\newcommand{\BB}{\mathcal{B}}
\newcommand{\EE}{\mathcal{E}}

\newcommand{\MM}{\mathcal{M}}

\newcommand{\KK}{\mathcal{K}}
\newcommand{\DD}{\mathcal{D}}

\newcommand{\calL}{\mathcal{L}^1}

\newcommand{\Lbase}{L_{|D}}
\newcommand{\Lo}{\psi(-L_{|D})}
\newcommand{\Los}{\psi^*(-L_{|D})}
\newcommand{\LoInv}{\psi(-L_{|D})^{-1}}

\newcommand{\INTFR}{\left((-\Delta)^{\beta/2}_{|D}\right)^{\alpha/2}}

\newcommand{\GDP}{G_D^\psi}
\newcommand{\GDPs}{G_D^{\psi^*}}

\newcommand{\PDS}{P^{\psi}_{ D}}

\newcommand{\GD}{G_{ D}}

\newcommand{\diam}{\textrm{diam}}
\newcommand{\dist}{\textrm{dist}}
\newcommand{\uu}{\textswab{u}}

\numberwithin{equation}{section}

\begin{document}
	\maketitle

    \begin{abstract}
    We study semilinear non-local elliptic problems driven by spectral-type operators of the form $\psi(-L_{|D})$ in a bounded $C^{1,1}$ domain $D\subset \R^d$ with a nonhomogeneous boundary condition. Here $\psi$ is a Bernstein function satisfying a weak scaling condition at infinity, and $L_{|D}$ is the generator of a killed L\'evy process. This general framework covers and extends the theory of the interpolated fractional Laplacian. A key novelty in this setting is the analysis of the nonhomogeneous boundary condition formulated in terms of the Poisson potential with respect to the $d-1$ Hausdorff measure on $\partial D$. We establish sharp boundary estimates for Green and Poisson potentials, introduce a weak $L^1$ trace-like boundary operator, and provide existence results for solutions under quite general nonlinearities, including sign-changing and non-monotone cases. The methodology combines stochastic process techniques, potential theory, and spectral analysis, and expresses the boundary behavior of the solution in terms of the renewal function and the distance to the boundary, suggesting a possible unified treatment of semilinear boundary problems in non-local settings.
\end{abstract}


	
	\section{Introduction}
	In this article, we solve the nonhomogeneous non-local semilinear problem of the elliptic type in a bounded $C^{1,1}$ domain $D\subset \R^d$:
    \begin{align}
			\begin{array}{rcll}\label{intro-eqn}
				\Lo u(x)&=&f(x,u(x)),&\textrm{in $D$},\\
				\frac{u}{\PDS\sigma}&=&\zeta,&\textrm{on $\partial D$}.
			\end{array}
	\end{align}
    Here $\psi$ is a Bernstein function satisfying a certain weak scaling condition, the operator $\Lbase$ is a non-local operator which is a generator of the L\'evy process killed upon exiting the domain $D$, while the function $\PDS \sigma$ is the reference function for the behaviour of the solution at the boundary.

    The best-known example of the operator covered by our setting is the so-called interpolated fractional Laplacian $\Lo=\INTFR$ where $\psi(\lambda)=\lambda^{\alpha/2}$, $\alpha\in (0,2)$, and where $\Lbase=-(-\Delta)_{|D}^{\beta/2}$, $\beta\in (0,2)$, is the restricted $\frac\beta2$-fractional Laplacian. However, our setting covers a much wider class of operators, including the fractional relativistic operators, various tempered operators, etc. For more details, we refer the reader to the condition \ref{as:WSC} and the discussion in the paragraph below it.

    The goal of this article is to take a step beyond the fractional framework and extend results to more general non-local operators, while at the same time covering the fractional case. The novelty of this article is the nonhomogeneous boundary theory in a quite broad generality, extending the interpolated fractional setting: the function $\psi$ is a complete Bernstein function that satisfies a weak scaling condition at infinity, and the operator $\Lbase$ is a generator of a killed subordinate Brownian motion, i.e.~$\Lbase=-\phi(-\Delta)_{|D}$, where the functions $\phi$ also satisfies the weak scaling condition at infinity. Moreover, even in the interpolated fractional Laplacian case, our results on the existence of solutions to the semilinear problem with superlinear functions $f$ are new.

    It is a non-local phenomenon that even the solutions to the linear equations blow up at the boundary - a behaviour which is also exhibited in our setting. By establishing the boundary behaviour of Green and Poisson potentials, we arrive at the abstract boundary condition $\frac{u}{\PDS\sigma}=\zeta$, which can be considered in a pointwise sense for sufficiently regular solutions, as a limit at the boundary, while in general it can be viewed as a weak $L^1$  trace-like boundary operator. Here $\PDS (x,z)$ is the Poisson kernel of $\Lo$ (i.e.~the generalized normal derivative of the Green function of $\Lo$) which is the cornerstone of the integral representation of harmonic functions with respect to $\Lo$, while $\PDS\sigma$ denotes the Poisson potential of the $(d-1)$ dimensional Hausdorff measure on $\partial D$ -- the function that describes the rate at which harmonic functions blow-up at the boundary.

    In this setting, we assume that the nonlinearity $f:D\times \R\to \R$ satisfies $|f(x,t)|\le q(x)\Lambda(t)$, for a locally bounded function $q$ and a non-decreasing function $\Lambda$. This allows us to deal with positive, negative, sign-changing, monotone, and non-monotone nonlinearities. The main goal of this article is to show the existence and continuity properties of solutions to \eqref{intro-eqn} for this wide class of nonlinearities.

    The operator $\Lo$ can be written in its spectral form
    \begin{align*}
        \Lo u=\sum_{j=1}^\infty  \psi(\lambda_j)\wh u_j \varphi_j,
    \end{align*}
    where $(\lambda_j,\varphi_j)$ are the eigenpairs of $\Lbase= - \phi(-\Delta)_{|D}$. However, this definition requires highly integrable data ($L^2$ theory), so it is less restrictive to look at the operator in its principal value form:
    \begin{align*}
        \Lo u(x)= \textrm{P.V.}\int_D\big(u(y)-u(x)\big)J_D(x,y)dy+\kappa(x)u(x).
    \end{align*}
    Here, the jumping kernel exhibits the interior behaviour of the same order as the kernel $j_{\psi\circ\phi}$ of the full operator $\psi(-L)=-\psi\circ\phi(-\Delta)$. On the other hand, the jumping kernel $J_D$ always continuously vanishes at the boundary, with possible different types of nonlinear orders depending on the relationship of $\phi$ and $\psi$, see Remark \ref{r:Lo pointwise}. To illustrate this behaviour in the classical case, note that the jumping kernel of the restricted fractional Laplacian $-(-\Delta)_{|D}^{ \alpha/2}$ does not continuously vanish at the boundary, while the jumping kernel of the spectral fractional Laplacian $-(-\Delta_{|D})^{\alpha/2}$ vanishes at the boundary with a linear order. While the jumping kernel of the interpolated fractional Laplacian behaves like the kernel of the restricted (and spectral) fractional Laplacian in the interior, it decays at the boundary with a fractional order. This mixed jumps structure makes the interpolated non-local operators of spectral type particularly relevant for applications in fractional modeling. 

       \subsection*{State-of-the-art}

    In recent years, significant progress has been made in semilinear elliptic problems involving non-local operators. The research mostly focuses on fractional-type operators and their generalizations in various directions: for existence results in fractional settings see e.g. \cite{Aba15a,BC17,BCBF18,bogdan_et_al_19,CFQ,chen_veron,dhifli2012,F20,FQ12}, for more general settings see e.g. \cite{BVW-na21,BJ20,Bio23cpaa,HuynhNguyen2022_new, KlimsiakRozkosz25}; for regularity see e.g. \cite{BMS-boundaryreg-23,grubb2016,Ros-OtonSerra2016}; for very large solutions see e.g. \cite{Aba17,BCBF16}; for potential-theoretic aspects see e.g. \cite{ksv_minimal2016,ksv2020boundary,Bio21jmaa,SongVondracek2006potenSpecial}; and some other connected research \cite{FJ23,SV2014,CKSV22,BW25,CVW25}.

    In what follows, we briefly present the principal references that are relevant for our work. In \cite{CGcV-jfa21} Chan, G\'omez-Castro and V\'azqez  gave a unified approach to the large eigenvalue problem (i.e.~for the source term $f(x,t)=\lambda t$) with nonhomogeneous boundary condition for a wide class of fractional-type operators in a domain. The framework included the restricted fractional Laplacian and the interpolated fractional Laplacian as examples, as well as a certain class of Schr\"odinger operators. Their approach assumed sharp bounds and controlled behavior of the Green function (in terms of fractional powers), both in the interior and on the boundary, as well as the existence of the Poisson/Martin kernel (the hardest assumption to verify, see \cite[Remark 2.5]{CGcV-jfa21}), from which they develop integrability theory and obtain large eigenvalues. In our approach, by finding the sharp behaviour of the Green potential  and proving the non-trivial result on the existence of the Poisson kernel, we are able to extend part of their results to  a wider class of interpolated non-local operators and non-linear source terms. Furthermore, the preparatory work has been done which should enable us to move to Schr\"odinger type problems in future work.  

    In the work by Huynh and Nguyen \cite{HuynhNguyen2022_new}, similarly to \cite{CGcV-jfa21}, the authors gave a unified approach to homogeneous semilinear problems for fractional-type operators (including the interpolated fractional Laplacian), again starting from the assumption on the sharp behaviour of the Green function. Although their source term $f$ is a quite general non-decreasing function, unlike \cite{CGcV-jfa21}, the authors consider only the semilinear problems with homogeneous boundary conditions. This simplifies the theory, since there is no involvement of the Poisson potentials, i.e., harmonic functions. On the other hand, our focused approach covers more general nonlinearities, as well as the nonhomogeneous boundary condition. Moreover, by using the general Dirichlet form theory instead of the fractional-type results, we are able to extend Kato's inequality from \cite{HuynhNguyen2022_new} to our setting. 
    
    The influential work by Bonforte, Figalli and V\'azquez \cite{BFV-cvPDE18} also employed a unified approach within the fractional framework, including the case of the interpolated fractional Laplacian. Their results addressed the homogeneous boundary condition, with a power-like nonlinearity $f(x,t)\sim t^p$. Furthermore, they studied H\"older regularity of solutions and obtained qualitative boundary estimates, as well as the change of the regularity behaviour in different regimes with respect to the order of the nonlinearity. It is worth noting that obtaining exact regularity up to the boundary, even in the setting of the interpolated fractional Laplacian, would be of significant interest. This is because the standard tool for establishing such results is the corresponding boundary Harnack principle, which in this setting does not have to hold, see \cite{ksv2020boundary}.
    
    Another quite general approach to semilinear non-local problems was introduced by Klimsiak and Rozkosz in \cite{KlimsiakRozkosz25}, where they used a probabilistic (martingale) method, and orthogonal projections to solve general semilinear problems in locally compact separable spaces. Although they cover a large class of operators, without the knowledge of the structure of harmonic functions with respect to the operator, their unified theory covers only the homogeneous boundary condition $\zeta=0$ in \eqref{intro-eqn}.
    
    A general approach to semilinear problems, covering a wide range of elliptic and parabolic differential operators of second order, as well as integro-differential operators, using balayage techniques has been presented recently by Bogdan and Hansen in \cite{BogdanHansen}, where they focus on the existence of positive solutions with increasing nonlinearities. Although they cover a wide range of existence results, our focus here lies more on fine properties of the explosion rates of solutions, even without the monotonicity assumption.

    To develop our theory for the operator $\Lo$ in this generality, several key preparatory results were essential. First, we emphasize the work of Biswas and L\"orinczi \cite{BL-na21} where they proved Hopf's lemma for $\Lbase$, providing the sharp boundary behaviour of the first eigenfunction which drives the behaviour of many phenomena. Furthermore, in \cite{KKLL-jfa19}, the authors established the generalized H"older regularity up to the boundary for the Green potentials of $\Lbase$, while in \cite{ksv_minimal2016}, Kim, Song, and Vondraček investigated potential-theoretic properties of the process generated by $\Lo$.

    Our approach to studying semilinear problems is motivated by the technique used for the spectral fractional Laplacian in \cite{AbaDup-Nonhomog2017}, and for more general non-local spectral-type operators in \cite{Bio23cpaa}. Moreover, the techniques for obtaining boundary behavior of Green and Poisson potentials in this generality are adapted from \cite{BVW-na21}, where the semilinear problem was studied for operators more general than the restricted fractional Laplacian. 
    We also note \cite{Bio21jmaa}, where the representation of harmonic functions for $\Lbase = -\phi(-\Delta)_{|D}$ was studied. The results from that work play a crucial role in our analysis, as we rely on them to establish the integral representation of harmonic functions with respect to $\Lo$, as well as to demonstrate the equivalence between the analytic and probabilistic notions of harmonicity.
    \subsection*{Main results}
        To make the exposition clearer, here we state claims only for the special case of the interpolated fractional Laplacian $\Lo=\INTFR$.
        
        \noindent\textbf{Semilinear problems:} The main focus of this article is the existence and uniqueness of weak-dual solutions to semilinear problems in Section \ref{s:semilinear}. The cornerstone result is a method of sub/supersolutions for the homogeneous problem (Theorem \ref{t:super-subsol}). 
        To investigate the nonhomogeneous problem for non-positive nonlinearities, we use a potential-theoretic approach of harmonic function approximation by an increasing sequence of Green potentials, see Theorem \ref{t:non-positive semilinear problem}. In the fractional setting, Theorem \ref{t:frac-semi-neg} shows that for $f(x,t)=-\de(x)^\theta|t|^p$, we can find the solution to the semilinear problem \eqref{intro-eqn} if and only if $p<\frac{1+\frac{2\theta}{2+\beta}}{1-\frac{\beta\alpha}{2+\beta}}$. In order to apply this potential-theoretic approach, we established the required representation of harmonic functions and the equivalence with the probabilistic notion of harmonicity, see Theorem \ref{t:1627}. Furthermore, the nonhomogeneous problem with non-negative nonlinearities was investigated in its full generality in Theorem \ref{t:semilin non-negative monotone linearity}. This result, translated into the fractional setting in Theorem \ref{t:frac-semi-pos}, for $f(x,t)=m\de(x)^\theta|t|^p$, implies the existence of a solution if and only if $p<\frac{1+\frac{2\theta}{2+\beta}}{1-\frac{\beta\alpha}{2+\beta}}$ for small enough $m>0$. The general nonhomogeneous problem with a signed nonlinearity was considered in Theorem \ref{t:semilinear signed data}, with its fractional  analogue given in Theorem \ref{t:frac-semi-sign}.

        \noindent\textbf{The boundary condition:} 
        In order to state the nonhomogeneous semilinear problem, it was essential to show that the Poisson kernel  $\PDS(x,z)$, given as the generalised normal derivative of the Green function $-\partial_V\GDP(x,z)$, is well-defined, see Proposition \ref{p:Poisson kernel}. Such results are case-specific and highly non-trivial, and generally represent one of the assumptions in recent works applying a unified approach (e.g. \cite{CGcV-jfa21,AGCV19}). The boundary condition is then given in terms of the reference function  $\PDS \sigma$, whose boundary behaviour is expressed in Proposition \ref{p:poisson bndry and finiteness} via the corresponding renewal function $V$ and distance to the boundary $\de$. 
        
        Furthermore, we provide the optimal weighted $L^1$ space, given in terms of the renewal function $V$, for the distributional formulation of the semilinear problem \eqref{intro-eqn}, see Subsection \ref{ss:homo-sol}, \eqref{eq:dist-weight-rho} and Definition \ref{r:defn of distributional solution}. The aforementioned results, together with the analysis of the weak boundary behaviour of general Poisson and Green potentials in Propositions \ref{p:bnd-bhv-PDSzeta} and \ref{p:boundary operator GD lambda}, enabled us to formulate the boundary condition in \eqref{intro-eqn} in terms of the weak $L^1$ boundary trace-like operator.    

    \paragraph*{Organization of the article and notation.}
    In Section \ref{s:prelim}, we present preliminary, mostly well-known, results on transition kernels, operators, and Green functions. Section \ref{s:operator} is concerned with the operator $\Lo$ in the $L^2$ setting -- its basic properties, the regularity of corresponding eigenvalues, and its principal value integral form. Section \ref{s:Green} deals with Green potentials and the optimal weighted space for studying distributional solutions to $\Lo u=f$. In Section \ref{s:Poisson harmonic} we studied the Poisson kernel, its potentials and their estimates, and the representation of harmonic functions with respect to $\Lo$. Section \ref{s:boundary} covers the weak boundary behaviour of potentials, as an introduction to the weak $L^1$ boundary trace-like operator. In Section \ref{s:linear problem} we give formal definitions of different notions of solutions to the nonhomogeneous problem and obtain Kato's inequality. Section \ref{s:semilinear} consists of the main existence results on semilinear problems. The article features an Appendix, with proofs of certain auxiliary claims based on already known techniques in the literature.
    	
	\noindent{\bf Notation:} For a continuous non-decreasing function $f:[0,\infty)\to [0,\infty)$ with $f(0)=0$, the space $C^f(D)$ denotes all functions $u:D\to \R$ with the finite norm
		\begin{align*}
			\|u\|_{C^f(D)}\coloneqq \| u\|_{L^\infty(D)}+\sup_{x,y\in D}\frac{|u(x)-u(y)|}{f(|x-y|)}<\infty.
		\end{align*}
		The space $C_0(D)$ denotes continuous functions in $D$ continuously vanishing at the boundary, $C_c(D)$ denotes those with a compact support, and $C_b(D)$ denotes continuous functions bounded on $D$. The space $C^\gamma(D)$ denotes locally H\"older continuous functions with the exponent $\gamma>0$, while the corresponding norm is given by $\|u\|_{C^f(D)}$ for $f(r)=r^\gamma$. The space $C^{1,1}(D)$ denotes differentiable functions on $D$ whose derivatives are locally Lipschitz. The space $C^k(D)$ denotes $k$-times differentiable functions in $D$. The spaces $L^2(D)$ and $L^\infty(D)$ denote usual Lebesgue spaces on $D$, where we write e.g. $L^2(D,\mu)$ if we use measure $\mu$ instead of the Lebesgue measure.
        
        The notation $\la u,v\ra$ denotes the scalar product in $L^2(D)$. By $B(x,r)$ we denote the ball centered at $x\in \R^d$ with a radius $r>0$, and the Euclidean norm by $|x|$. The function $\de(x)=\inf\{|x-z|:z\notin D\}$ denotes the distance to $D^c$ of a point $x$ (i.e.~distance to the boundary if $x\in D$). For a set $A$, the notation $\overline A$ denotes its closure, while $U\subsub D$ means that $U$ is compactly contained in $D$.
		
        For two non-negative functions $f$ and $g$, $f \lesssim g$ in $D$ means that there exists $c>0$ such that $f(x)\le c\,g(x)$, $x\in D$, while $f\asymp g$ means $f\lesssim g$ and $g\lesssim f$. The constants are denoted by small letters $c$, $c_1$, $c_2$, and may change their value from line to line. More important constants are denoted by a big letter $C$, where e.g. $C(a,b)$ means that the constant depends on the values of $a$ and $b$. All functions considered in the article are assumed to be Borel measurable.
    
	\section{Preliminaries}\label{s:prelim}
	
	Let $W=(W_t)_{t\ge 0}$ be a Brownian motion in n $\R^d$, $d\ge2$, with the characteristic exponent $\xi\mapsto |\xi|^2$. Let $S=(S_t)_{t\ge 0}$ be a subordinator, independent of $W$, with the Laplace exponent $\phi$, i.e.~$S$ is a non-negative (i.e.~non-decreasing) L\'evy process such that
    \begin{align*}
        \ex\left[e^{-\lambda S_t}\right]&=e^{-\phi(\lambda)t},\quad \lambda\ge 0,\,t\ge 0.
    \end{align*}
    The exponent $\phi$ is a Bernstein function of the form
    \begin{align*}
        \phi(\lambda)&=b\lambda +\int_0^\infty (1-e^{-\lambda t})\mu(dt),\quad \lambda\ge 0,
    \end{align*}
    and completely determines the distribution of the subordinator $S$, for details see \cite{bernstein}. The parameter $b\ge 0$ is called the drift and the non-negative measure $\mu$, satisfying the condition $\int_0^\infty (1\wedge t)\mu(dt)<\infty$, is called the L\'evy measure of the subordinator.

    The process $X=(X_t)_{t\ge 0}$ defined by $X_t=W_{S_t}$, called the subordinate Brownian motion, is a L\'evy process with the characteristic exponent $\xi \mapsto \phi(|\xi|^2)$. Throughout the article, we will assume that $b=0$, i.e. $S$ is a driftless subordinator, which implies that $X$ is a pure-jump L\'evy process. Mo rover, we assume:
    \begin{assumption}{A}{1}\label{as:WSC}
        The exponent $\phi$ is a complete Bernstein functions, i.e.~the L\'evy measure $\mu$ has a completely monotone density denoted by $\mu(t)$, and $\phi$ satisfies the weak scaling condition at infinity: there exist $\ua{\phi},\oa{\phi}>0$ and $0< \ud{\phi}\le \od{\phi}<1$ such that
        \begin{align}\label{eq:WSC}
            \ua{\phi} \lambda^{\ud{\phi}} \phi(t)\le \phi(\lambda t)\le \oa{\phi}
            \lambda^{\od{\phi}}\phi(t),\quad \lambda\ge 1, \, t\ge 1.
        \end{align}
    \end{assumption}
    This assumption drives the small space-time behaviour of $X$ and is often one of the standard assumptions in many state-of-the-art articles dealing with the potential theoretical properties of L\'evy processes. We note that the main example of a process for which \ref{as:WSC} holds is the isotropic $\alpha$-stable process, i.e.~$\phi(\lambda)=\lambda^{\alpha/2}$, for some $\alpha\in (0,2)$, which satisfies exact scaling at both infinity and zero. Other examples of such subordinators, such as the relativistic stable, or the tempered stable, and others, can be found in \cite[p. 409]{KSV_bhpinf-PA14} and \cite[p. 41]{KSV_heat-PA18}.

    The main analytical object in our study is the operator $\Lo$, which we introduce in Section \ref{s:operator}. The underlying stochastic process $Y$ is obtained by appropriate killing and subordination of the Brownian motion. First, we recall the notion of a stochastic process killed upon exiting a set. Let $D\subset \R^d$ be an open set, and define $\tau_D\coloneqq \inf\{t>0:X_t\in D^c\}$ the first exit time of $X$ from $D$. The killed process $X^D=(X^D_t)_{t\ge 0}$ of $X$ is defined by
    \begin{align*}
        X^D_t\coloneqq \begin{cases}
            X_t, &t<\tau_D\\
            \partial, &t\ge \tau_D,
        \end{cases}
    \end{align*}
    where $\partial$ is called the cemetery and it is a point added to $\R^d$. 

    Let $T=(T_t)_{t\ge t}$ be another subordinator independent of $W$ and $S$ 
    (and thus of $X$ and $X^D$), with the Laplace exponent $\psi$ of the form 
     \begin{align*}
        \psi(\lambda)&=\int_0^\infty (1-e^{-\lambda t})\nu(dt),\quad \lambda\ge 0.
    \end{align*}
   Again, we assume that $\psi$ satisfies \ref{as:WSC} with some parameters $\ua{\psi},\oa{\psi}>0$, and $0<\ud{\psi}\le \od{\psi}<1$.

        The main process of interest $Y=(Y_t)_{t\ge 0}$ is the randomly time-changed process $X^D$ by $T$, defined by $Y_t=(X^D)_{T_t}$. This process belongs to a large class of subordinate killed L\'evy processes and is associated with the operator $\Lo$ as its infinitesimal generator, Section \ref{s:operator}.

    Although many of the stated results hold for more general domains, from now on, we always assume that $D$ is a $C^{1,1}$ bounded open set.
       
    \subsection*{Transition kernels}   
    
    The process $X$ has a density $p(t,x,y)$, which due to subordination has the representation 
    \begin{align}\label{eq:SBM dens}
        p(t,x,y)=\ex\left[\frac{1}{(4\pi S_t)^{d/2}}e^{-\frac{|x-y|^2}{4S_t}}\right],\quad x,y\in \R^d,\, t>0.
    \end{align}
    In particular, $p(t,x,y)$ is symmetric around the
	origin, i.e. $p(t,x,y)=p(t,x-y)$. Under the condition \ref{as:WSC} the small-time heat kernel estimates are given by
	\begin{align*}
		p(t,x,y)\asymp \frac{1}{\phi^{-1}(1/t)^{d/2}}\wedge
		\frac{t\phi(|x-y|^{-2})}{|x-y|^{d}},\quad |x-y|<M, \, t\in (0,1),
	\end{align*}
		see \cite{Mimica-heatSBM-proceed16} or \cite{bogdan_density_and_tails_unimodal}, where the comparability constants depend on $\phi$, $M$ and $d$.
	In fact, for all $x,y\in \R^d$ it holds that 
	\begin{align}\label{eq:den-tails bnd}
		p(t,x,y)\le C(d)t\frac{\phi(|x-y|^{-2})}{|x-y|^d},\quad t>0,
	\end{align}
	see \cite[Corollary 7]{bogdan_density_and_tails_unimodal}.
    The $C_0(\R^d)$-semigroup generated by $X$ is defined by
    \begin{align*}
        P_tf(x)=\ex_{x}[f(X_t)]=\int_{\R^d} p(t,x,y)f(y)dy,\quad x\in \R^d,\,t>0,
    \end{align*}
    with the infinitesimal generator given by 
        \begin{align}
            Lu(x)&=\int_{\R^d}\big(u(y)-u(x)-\nabla u(x)\cdot (y-x)\1_{|x-y|<1}\big)j_\phi(|x-y|)dy\nonumber\\
            &=\lim_{\varepsilon\downarrow 0}\int_{|x-y|> \varepsilon}\big(u(y)-u(x)\big)j_\phi(|x-y|)dy\nonumber\\
            &=:\textrm{P.V.}\int_{\R^d}\big(u(y)-u(x)\big)j_\phi(|x-y|)dy,\label{eq:L-defRd}
        \end{align}
        for all $u\in C_0^2(\R^d)$, where $j_\phi(|x-y|)$ is the L\'evy density of the L\'evy measure of $X$, given by
        \begin{align}\label{eq:jumping jphi}
            j_\phi(r)=\int_0^\infty \frac{1}{(4\pi t)^{d/2}}e^{-\frac{r^2}{4t}}\mu(t)dt,\quad r>0,
        \end{align}
        see \cite[Chapter 6]{sato}.
        We extend the definition of $L$ to all functions for which the expression in \eqref{eq:L-defRd} is well defined, which is certainly true for all $C^{1,1}(\R^d)\cap L^1(\R^d,(1\wedge j(|x|))dx)$ functions, see e.g. \cite[Section 3]{BogBycz-PoteThe1999}.

        The killed process $X^D$ has the transition density $p_D(t,x,y)$, which is given by Hunt's formula: for all $t>0$ and $x,y\in \R^d$,
	\begin{align}\label{eq:p_D Hunt}
		p_D(t,x,y)=p(t,x,y)-\ex_x[p(t-\tau_D,X_{\tau_D},y)\1_{\{t<\tau_D\}}].
	\end{align}
        The transition density is symmetric and jointly continuous up to the boundary, see Lemma \ref{ap:l:joint p_D}. Further, under \ref{as:WSC} the sharp estimates of  $p_D(t,x,y)$ are obtained in
	\cite{KM-ejp18-heatkerneldomain}:  for $x,y\in D$  and small time $t<T$
	it holds that
	\begin{align}\label{eq:heat-sharp-small}
		\begin{split}
			p_D(t,x,y)&\asymp \left(\frac{V(\de(x))}{\sqrt{t}}\wedge
			1\right)\left(\frac{V(\de(y))}{\sqrt{t}}\wedge
			1\right)\times\\
			&\hspace{6em}\left({\phi^{-1}(1/t)^{d/2}}\wedge
			\frac{t}{|x-y|^{d}V(|x-y|)^2}\right),
		\end{split}
	\end{align}
	where $V$ is the renewal function of $X$. We refer to \cite[Subsection 4.1]{BVW-na21} for the construction of the function $V$ and emphasize its important role in various boundary phenomena, which will also later be demonstrated in the article. Under \ref{as:WSC}, $V$ satisfies
    \begin{align*}
	    V(t)\asymp \frac{1}{\sqrt{\phi(t^{-2})}},\quad t<1,
    \end{align*}
    which implies that $V$ also satisfies the weak scaling condition \eqref{eq:WSC}, but at zero:
    \begin{align}\label{e:wsc-V}
    	\wt{a}_1 \left(\frac{t}{s}\right)^{\ud{\phi}} \le \frac{V(t)}{V(s)} \le \wt{a}_2 \left(\frac{t}{s}\right)^{\od{\phi}} , \quad 0<s\le t\le 1.
    \end{align}
    
    The large-time sharp bounds for the transition density $p_D(t,x,y)$ can be given in terms of the smallest eigenvalue $\lambda_1>0$ of $L$,
	\begin{align}\label{eq:heat-sharp-big}
		p_D(t,x,y)&\asymp e^{-\lambda_1 t}V(\de(x))V(\de(y))\quad x,y\in D,
	\end{align}
see Section	\ref{s:operator}. The killed process $X^D$ has a $C_0(D)$-semigroup given for $x\in D$, $t>0$, and $f\in C_0(D)$ by
    \begin{align*}
        P^D_tf(x)=\ex_{x}[f(X^D_t)]=\ex_{x}[f(X_t);t<\tau_D]=\int_D p_D(t,x,y)f(y)dy, 
    \end{align*}
    where we naturally extend the definition of a function $f:D\to\R$ to $\partial$ by $f(\partial)=0$. The semigroup $(P_t^D)_t$ has the strong Feller property, i.e. $P_t^D L^\infty(D)\subset C_b (D)$, which follows by continuity of $p_D$. Since $X^D$ is a symmetric process, it is easy to see that the semigroup $(P^D_t)_t$ is self-adjoint. The associated infinitesimal generator to $(P^D_t)_t$ is denoted by $\Lbase$ and for $u\in \DD(\Lbase)$ it holds:
    \begin{align}\label{eq:Lbase-defD}
        \Lbase u(x)=Lu(x)=\textrm{P.V.}\int_D\big(u(y)-u(x)\big)j_\phi(|x-y|)dy-u(x)\kappa(x),
    \end{align}
    where the function $\kappa(x)\coloneqq\int_{D^c}j_\phi(|x-y|)dy$ is called the killing density for $\Lbase$. Here the $C_0(D)$ domain of $\Lbase$ is the space of all $C_0(D)$ functions such that the expression above is in $C_0(D)$, see \cite{BaeumerLuksMeersc-SpaceTime18} for a delicate study of the domain of the infinitesimal generators of killed Feller processes. We note that $C_c^2(D)$ is not a subset of the domain $\DD(\Lbase)$ in the $C_0(D)$-sense by Appendix \ref{ap:domain}. However, there is a unique extension of $(P^D_t)_t$ to a $L^2(D)$-semigroup, for which $C_c^\infty(D)$ functions are in the $L^2(D)$ domain of the infinitesimal generator. These two facts are generally known, but due to the lack of a formal reference, we present a short proof in Lemma \ref{ap:l:generator domain}. Nevertheless, throughout this paper we use the same notation for the semigroups (as well as their infinitesimal generators) in both spaces, when appropriate, since they are equal a.e.~on $D$.

    Using the subordination principle, the transition density of $Y_t=(X^D)_{T_t}$, denoted by $r_D(t,x,y)$, is given by
        \begin{align*}
            r_D(t,x,y)=\ex[p_D(T_t,x,y)]=\int_0^\infty p_D(s,x,y)\p(T_t\in ds),\quad x,y\in D, \, t>0.
        \end{align*}
    The semigroup of $Y$ (both in the $C_0(D)$ and $L^2(D)$ sense) is denoted by $(R^D_t)_t$, and can be viewed as the Bochner subordinated semigroup $(P^D_t)_t$. We also note that $R^D_t$ is symmetric and strongly Feller, see \cite[Proposition V.3.3]{bliedtner}.
    
    \subsection*{Green functions}
    Since  $D$ is bounded, the process $X^D$ is transient and has the Green function $G_D$  given by
	\begin{align*}
		G_D(x,y)&=\int_0^\infty p_D(t,x,y)dt.
	\end{align*}
	The upper bound for $G_D$ is known: for $x,y\in D$ it holds that
	\begin{align}\label{eq:GD sharp bound SBM}
		G_D(x,y)\lesssim 
		\frac{1}{|x-y|^d V(|x-y|)^{-2}},
	\end{align}
	where the constant depends on $d$, $\phi$ and $D$, see \cite[Lemma 2.1]{ksv2020boundary}. We note that $G_D$ is finite off the diagonal,  infinite on the diagonal, and jointly continuous in the extended sense, see Lemma \ref{l:GDP cont}. The corresponding Green potential operator has the following probabilistic interpretation
    \begin{align*}
    	G_D f(x)\coloneqq \int_D G_D(x,y)f(y)dy=\ex_{x}\left[\int_0^{\tau_D}f(X_t)dt\right],\quad f\ge 0.
    \end{align*}

	The Green function of $Y$ is also well defined and is given by
        \begin{align}\label{eq:GDP_representation}
		\GDP(x,y)\coloneqq \int_0^\infty r_D(t,x,y)dt=\int_0^\infty p_D(t,x,y)\uu(t)dt,
	\end{align}
        where $\uu$ denotes the density of the potential measure of $T$, i.e. 
        $U(A)\coloneqq  \ex\left[\int_0^\infty \1_A(T_t)dt\right]=\int_A \uu(t)dt$, $A\in \BB([0,\infty))$. Note that the existence of a non-increasing potential density $\uu$, integrable on bounded sets, follows from the assumption that $\psi$ is a complete Bernstein function, and in particular a special Bernstein function, see \cite[Chapter 11]{bernstein}.
        
        As before, the probabilistic interpretation of the Green operator $\GDP$ is
        \begin{align*}
        	\GDP f(x)\coloneqq \int_D \GDP(x,y)f(y)dy=\ex_{x}\left[\int_0^{\infty}f(Y_t)dt\right],\quad f\ge 0,
        \end{align*}
        and the sharp bounds for $\GDP$ are known for all $x,y\in D$
	\begin{align}\label{eq:GDP sharp}
		\GDP(x,y)\asymp \left(\frac{V(\de(x))}{V(|x-y|)}\wedge 1\right)\left(\frac{V(\de(y))}{V(|x-y|)}\wedge 1\right)\frac{1}{|x-y|^{d}\psi(V(|x-y|)^{-2})},
	\end{align}
	see \cite[Theorem 6.4]{ksv2020boundary}.
    By standard calculation, it is easy to show that $\GDP$ is jointly continuous in the extended sense in $D\times D$, see Lemma \ref{l:GDP cont}.    

    There is a strong connection between the Green functions of $X^D$ and $Y$, i.e. between $G_D$ and $\GDP$. In order to establish this connection, we introduce the conjugate Bernstein function $\psi^*(\lambda)\coloneqq\frac{\lambda}{\psi(\lambda)}$ of $\psi$, for details see \cite{bernstein}. Since $\psi$ satisfies \ref{as:WSC}, the function $\psi^*$ satisfies
	\ref{as:WSC} with constants $\ua{\psi^*}=1/\oa{\psi}$, $\oa{\psi^*}=1/\ua{\psi}$, and $\ud{\psi^*}=1-\oa{\psi}$ and $\od{\psi^*}=1-\ud\psi$. The Bernstein function $\psi^*$ generates a subordinator $T^*$, called the conjugate subordinator of $T$, and we can consider the stochastic process $Y^*$ (and the accompanying potential theoretical objects) obtained by subordinating the process $X^D$ by $T^*$. As before, denote the Green function of the $Y^*$ by $\GDPs$ and the potential density of $T^*$ by $\uu^*$. One can show, see for example \cite[Proposition 14.2.(ii)]{bernstein}, the following interplay of the Green functions $\GDP$, $\GDPs$ and $G_D$.
	\begin{lem}\label{l:compos-of-GDs}
		For every $x,y\in D$ it holds that
		\begin{align*}
			\int_D \GDP(x,\xi)\GDPs(\xi,y)d\xi=G_D(x,y).
		\end{align*}
	\end{lem}

	\section{Operator $\Lo$}\label{s:operator}
	
	In \cite[Theorem 2.1]{ChenSong-Spectral07} it was shown that the operators $P^D_t$ are compact for all $t>0$, hence they have a discrete spectrum, which implies that there is an orthonormal basis $(\varphi_j)_j$ on $L^2(D)$ consisting of
	eigenfunctions of $\Lbase$ (as well as $P^D_t$):
	\begin{align}\label{1028}
		\Lbase \varphi_j=-\lambda_j\varphi_j,\quad \text{in $L^2(D)$,}
	\end{align}
	where $0<\lambda_1\le \lambda_2\le\lambda_3\le\dots$ satisfy
	\begin{align*}
		\lambda_j\asymp \phi(j^{2/d}),\quad j\in \N.
	\end{align*}
	The last assertion holds by \cite[Theorem 4.5]{ChenSong-Two-sided-JFA05} and the fact that the eigenvalues of the Dirichlet Laplacian $\Delta_{|D}$ behave like $j^{2/d}$ by Weyl's law, see also \cite[Theorem 2.7]{chung_zhao}.
	
	In the case of the fractional or classical Laplacian, the eigenfunctions $\varphi_j$ are bounded in $D$ and sufficiently regular in the suitable H\"older space, even up to the boundary, with the bounds given in terms of the eigenvalues. When $L$ is the fractional Laplacian, this was proved in \cite{Fernandez-RealRos-Oton2014}. Following the approach in \cite{Fernandez-RealRos-Oton2014}, in Appendix \ref{ap:GR-reg} we prove the following analogous result in our more general non-local setting and for generalized H\"older spaces.
	 	 
	\begin{thm}\label{t:eigenfuncions}
		For every $j\in \N$ we have $\varphi_j\in C_0(D)\cap C^V(D)$ and there exist
		constants $C_1,\, C_2,\,C_3,\,k,\,\varepsilon>0$, depending only on $d$, $D$ and $\phi$, such that
		\begin{align}
			\|\varphi_j\|_{L^\infty(D)}&\le C_1\lambda_j^{k-1},\label{eq:eigen no0}\\
			\|\varphi_j\|_{C^{V}(D)}&\le C_2\lambda_j^{k},\label{eq:eigen no1}\\
			\left\|\frac{\varphi_j}{V(\de)}\right\|_{C^{\varepsilon}(D)}&\le
			C_3\lambda_j^{k}.\label{eq:eigen no2}
		\end{align}
	\end{thm}
		
	The principal eigenvalue $\varphi_1$ can be chosen such that $\varphi_1>0$, and then by \cite[Theorem 3.2 \& Theorem 3.3]{BL-na21} it holds that
	\begin{align}\label{eq:eigen sharp}
		\varphi_1\asymp V(\de)\quad \text{in $D$}.
	\end{align}
	Define the operator $\Lo$ by
	\begin{align*}
		\Lo u=\sum_{j=1}^\infty \psi(\lambda_j)\wh u_j \varphi_j,
	\end{align*}
	where $\wh u_j$ are the coefficients of $u\in L^2(D)$ with respect to the orthonormal basis $(\varphi_j)_j$, with domain
	\begin{align}\label{eq:domain L2 Lo}
		\DD(\Lo)=\{u=\sum_{j=1}^{\infty} \wh u_j \varphi_j\in L^2(D):
		\sum_{j=1}^\infty \psi(\lambda_j)^2|\wh u_j|^2<\infty\}.
	\end{align}
	It can be easily seen that the space $\DD(\Lo)$ is a Banach space equipped with the norm $\|u\|^2_{\DD(\Lo)}=\sum_{j=1}^\infty \psi(\lambda_j)^2|\wh u_j|^2$, and that $\Lo$ is an unbounded operator in $L^2(D)\to L^2(D)$. Further, $\Lo$ has a bounded inverse
	\begin{align*}
		\LoInv f=\sum_{j=1}^\infty \frac{1}{\psi(\lambda_j)}\wh f_j \varphi_j,\quad f \in L^2(D).
	\end{align*}
    In the next proposition, we prove that the operator can be seen as the infinitesimal generator of the $L^2(D)$ semigroup generated by $Y$.
	\begin{prop}\label{p:Lo inverse L2(D)}
		Let $f\in L^
        2(D)$. For a.e. $x\in D$, it holds that 
		\begin{align*}
			\LoInv f(x)=\GDP f(x).
		\end{align*}
        In particular, $\GDP f\in \LLL$.
	\end{prop}

    \begin{proof}
		The proof of this claim is very similar to \cite[Proposition 2.2]{Bio23cpaa} and \cite[Lemma 11]{AbaDup-Nonhomog2017} so we keep the proof concise and put the emphasis on the differences.
        
		First, for $f=\varphi_1\ge0$, by \eqref{eq:GDP_representation}, \eqref{1028}, and Tonelli's theorem,  we have for a.e. $x\in D$
		\begin{align*}
			\GDP\varphi_1(x)&=\int_0^\infty \uu(t)\int_D
			p^D(t,x,y)\varphi_1(y)dydt=\int_0^\infty \uu(t)e^{-\lambda_1
				t}\varphi_1(x)dt=\LoInv\varphi_1(x),
		\end{align*}
		where the last equality follows from the connection between the potential density and the Laplace exponent of the subordinator $T$, see \cite[Eq. (5.20)]{bernstein}.
		
        By \eqref{eq:eigen no2} and \eqref{eq:eigen sharp} we get $|\varphi_j|\lesssim V(\de)$, so by Fubini's theorem and the previous calculation, we obtain $\GDP f=\LoInv f$ a.e. in $D$ for all $f\in \textrm{span}\{\varphi_j:j\in \N\}$.

		The operator $\G$ on $\textrm{span}\{\varphi_j:j\in \N\}$ given by $\G f=\GDP f$ is an isometry on $\textrm{span}\{\varphi_j:j\in \N\}$, with the standard $L^2(D)$ norm in the domain, and the norm $\|\cdot\|_{\DD(\Lo)}$ in the codomain. Hence, $\G$ can be uniquely extended to an isometry on $L^2(D)$, which corresponds to $\LoInv$ since both operators agree on $\textrm{span}\{\varphi_j:j\in \N\}$.
        
        Take $f=\sum_{i=1}^\infty \wh f_j \varphi_j\in L^2(D)$ and the approximating sequence $f_n=\sum_{i=1}^n \wh f_j \varphi_j$, $n\in \N$. Obviously, $\G f_n \to \G f$ in $L^2(D)$, and by
        \begin{align*}
            \int_D\int_D \GDP(x,y)|f(y)-f_n(y)|dy\varphi_1(x)dx&=\frac{1}{\psi(\lambda_1)}\int_D|f(y)-f_n(y)|\varphi_1(y)dy\\
            &\le \frac{\| f-f_n\|_{L^2(D)}}{\psi(\lambda_1)},
        \end{align*}
        we have that $\GDP f = \G f =\LoInv f$ a.e. in $D$ and that $\GDP(x,\cdot)f(\cdot)\in \LLL$ since $\varphi_1\asymp V(\de)$ by \eqref{eq:eigen sharp}.
	\end{proof}

	\begin{rem}\label{r:infty-gener}
    \begin{enumerate}
        \item[(i)]
        We defined the operator $\Lo$ by subordinating the spectral decomposition of $\Lbase$. However, by recalling that the $R^D$ semigroup of $Y$ is the subordinated semigroup $P^D$ of $X^D$, and by noting that for all $j\in \N$, and  $t>0$ we have
        \begin{align*}
            R^D_t \varphi_j=\int_0^\infty P^D_s \varphi_j \p(T_t\in ds)=\int_0^\infty e^{-\lambda_j s}\p(T_t\in ds)=e^{-\psi(\lambda_j)t},
        \end{align*}
        it is easy to see that the $L^2(D)$ infinitesimal generator of $R^D$ has the same spectral form as $\Lo$. Thus, $\Lo$ is the infinitesimal generator of $Y$ in $L^2(D)$, with the domain $\DD( \Lo)$ as in \eqref{eq:domain L2 Lo}, and it can be viewed as the Bochner subordinated operator $\Lbase$. For an overview of the general theory, see e.g. \cite[Chapter 13]{bernstein}.  Further, the relation $\Lo \GDP f=f$ for $f\in L^2(D)$ from Proposition \ref{p:Lo inverse L2(D)} shows that $\GDP$ is indeed the potential operator of the infinitesimal generator of 
        $Y$ in the proper sense in $L^2(D)$.

        \item[(ii)]
        Since $\Lo$ can be viewed as the Bochner subordinated $\Lbase$, by  \cite[Theorem 13.6]{bernstein} it follows that $C_c^\infty(D)$ is in $\DD(\psi(-\Lbase))$, since $C_c^\infty(D)$ is in $L^2(D)$ domain of $\Lbase$, see  Lemma \ref{ap:l:generator domain}. In the case when $X^D$ is a local operator, e.g.~when $L=\Delta$, the spectral coefficients of $\varphi\in C_c^\infty(D)$ with respect to $(\varphi_j)_j$ exhibit fast decay, see \cite[Eq. (22)]{AbaDup-Nonhomog2017} for the spectral Laplacian and \cite[Eq. (2.26)]{Bio23cpaa} for a more general setting. However, the analogous calculation fails in the case when $L$ is a non-local operator. This fast decay of coefficients of $C_c^\infty(D)$ functions was heavily used in \cite{AbaDup-Nonhomog2017,Bio23cpaa}. In this paper, we proceed with an alternative approach, which allows us to obtain the analogous result without assumptions on the decay rate.  
    \end{enumerate}      
	\end{rem}

    Note that analogous results obtained for the operator $\Lo$ hold also for $\Los$, since $\psi$ and $\psi^*$ satisfy the same assumptions. Similarly to Lemma \ref{l:compos-of-GDs} which connects the Green operators, the following connection between the operators $\Lo$, $\Los$, and $\Lbase$ was proved in \cite[Corollary 13.25]{bernstein}.
    
	\begin{prop}\label{p:composition-Lo}
		For every $f\in L^2(D)$ which is in the $L^2(D)$-domain of the operator
		$\Lbase$, we have
		\begin{align*}
			\Lo \circ \Los f=-L_{|D} f,\quad \textit{a.e. in $D$.}
		\end{align*}
	\end{prop}
		
	In what follows, we obtain the integral representation of the operator $\Lo$, based on the representation of Bochner subordinated operators.
	\begin{prop}\label{p:Lo pointwise}
		For $u\in C^{1,1}(D)\cap \DD(\Lbase)$ it holds for almost all $x\in D$ that
		\begin{align}\label{eq:Lo point}
			\Lo u(x)=\textrm{P.V.}\int_D \big(u(x)-u(y)\big)J_D(x,y)dy+\kappa(x)u(x),
		\end{align}
		where 
		\begin{align*}
			J_D(x,y)\coloneqq \int_0^\infty p_D(t,x,y)\nu(t)dt,\quad\text{and}\quad
			\kappa(x)\coloneqq \int_0^\infty\left(1-P_t^D\1(x)\right) \nu(t)dt.
		\end{align*}
	\end{prop}
	
	\begin{proof}
		The proof follows the same approach as the one in \cite[Lemma 2.4]{Bio23cpaa}, where $L=\Delta$, and relies on the formula 
        \begin{align*}
            \Lo u = \int_0^\infty\big(  u-P_t^D u\big)\nu(t)dt,\quad u\in \DD(\Lbase).
        \end{align*}
        We can repeat all the calculations from \cite{Bio23cpaa} in the same way but instead of using the explicit form of the Brownian heat kernel as in \cite[Eq. (2.32)]{Bio23cpaa} we use the bound \eqref{eq:den-tails bnd} to obtain
		\begin{align*}
			p(t-\tau_D,X_{\tau_D},y)\lesssim 1\wedge t,\quad y\in B(x,1\wedge (\de(x)/4)), \quad \text{$\p_x$-a.s.}
		\end{align*}
	\end{proof}
	\begin{rem}\label{r:Lo pointwise}
    The boundary behaviour of the jumping density $J_D$ of $Y$ is crucial in discerning which weighted $L^1$ space should be considered for distributional solutions to equations related to $\Lo$. First we note that $J_D$ vanishes at the boundary. When $L=\Delta$, the sharp bound for $J_D$ from \cite{ksv_minimal2016,Bio23cpaa} is of the form
    \begin{align*}
        J_D(x,y)\asymp \left(\frac{\de(x)}{|x-y|}\wedge 1\right)\left(\frac{\de(y )}{|x-y|}\wedge 1\right) j_\psi(|x-y|),\quad x,y\in D.
    \end{align*}
     Although in Subsection \ref{ss:homo-sol} we describe the boundary behaviour of $J_D(x,y)$ for fixed $y\in D$, the sharp behaviour in both $x$ and $y$ is not known in the full generality of our setting. In essence, this complication occurs because both $\psi$ and $\phi$ make non-local contributions to $J_D$. The most significant progress on this problem was made in \cite{ksv2020boundary}, when $\psi$ satisfies \eqref{eq:WSC} with
	$\frac12<\ud{\psi}$ or $\od{\psi}<\frac12$, where analogous estimates can be obtained from \cite[Theorem 8.4]{ksv2020boundary}. Specifically, for $\psi(\lambda)=\lambda^{\alpha/2}$, $\alpha\in(0,2)$, by \cite[Example 8.5]{ksv2020boundary} for all $x,y\in D$ the kernel $J_D(x,y)$ is comparable to 
    \begin{align*}
        \begin{cases}
            \left(\frac{V(\de(x)\wedge \de(y))^2}{V(|x-y|)^2}\wedge
            1\right)^\frac12\left(\frac{V(\de(x)\vee \de(y))^2}{V(|x-y|)^2}\wedge
            1\right)^{\frac12-\frac{\alpha}2}\frac{1}{|x-y|^{d}
                V(|x-y|)^{\alpha}},\enskip\alpha<1,\\
            \left(\frac{V(\de(x)\wedge \de(y))^2}{V(|x-y|)^2}\wedge
            1\right)^\frac12\log\left(1+\frac{V(\de(x)\vee \de(y))^{2}\wedge
                V(|x-y|)^2}{V(\de(x)\wedge \de(y))^{2}\wedge
                V(|x-y|)^2}\right)\frac{1}{|x-y|^{d}
                V(|x-y|)},\enskip\alpha=1,\\
            \left(\frac{V(\de(x)\wedge \de(y))^2}{V(|x-y|)^2}\wedge
            1\right)^{1-\frac{\alpha}2}\frac{1}{|x-y|^{d} V(|x-y|)^{\alpha}},\enskip\alpha>1.
        \end{cases}
    \end{align*}

    It is also worth mentioning that the killing density $\kappa$ of $Y$ is continuous in $D$ and belongs to $\LLL$. The continuity follows from the fact that $X^D$ is strongly Feller, and from the bound \cite[Corollary 2.8]{barriers_BGR}, cf. \cite[Remark 2.5]{Bio23cpaa}. The integrability in $\LLL$ follows from \eqref{eq:eigen sharp} and $P^D_t\varphi_1=e^{-\lambda_1 t}\varphi_1$.
\end{rem}

	\section{Green potentials}\label{s:Green}

    In this section, we study the functions $u=\GDP f$, their integrability, and sharp boundary behaviour. The main goal is to prove that $\GDP f$ distributionally solves $\Lo u=f$ in the largest class of functions $f$ such that $\GDP f$ is well defined, see  Theorem \ref{t:GDPf dist-sol}.

	\begin{lem}\label{l:Green invariant}
		It holds that
		$$\GDP V(\de)(x)\asymp V(\de(x)), \quad x\in D.$$
		Furthermore, $\GDP f$ is finite a.e. in $D$ if and
		only if $f\in \LLL$. Additionally, there  exists $C=C(d,D,\phi,\psi)>0$ such that 
		\begin{align*}
			\| \GDP f\|_{\LLL}\le C\|f\|_{\LLL}, \quad f\in \LLL.
		\end{align*}
	\end{lem}
	\begin{proof}
By Proposition \ref{p:Lo inverse L2(D)}  we have that $\GDP
		\varphi_1=\psi(\lambda_1)^{-1}\varphi_1$, so the first claim follows from
		\eqref{eq:eigen sharp}. Moreover, the Fubini theorem then implies that if $f\in \LLL$, then $\GDP f\in \LLL$, hence $\GDP f$ is finite a.e. Assume now that $f\ge 0$ and $\GDP f(x)$ is finite at some point $x\in D$. Then, by the strong Markov property of $Y$, 
		\begin{align*}
			\GDP f(x)&\ge  
			 \ex_x\left[\ex_{Y_{ \tau_{B(x,\varepsilon)} } } \left[\int_0^\infty f(Y_t)dt\right]\right]=\int_{D\setminus B(x,\varepsilon)}\GDP f(y)\p_x(Y_{\tau_{B(x,\varepsilon)}}\in dy)
		\end{align*}
		for and $\varepsilon\in (0, \de(x))$. Since the distribution of $Y_{\tau_{B(x,\varepsilon)}}$ is positive and locally bounded from below for all $y\in D\setminus \overline{B(x,\varepsilon)}$, see \cite[p.~146-148]{ksv2020boundary}, $\GDP f$ is finite a.e.~on $y\in D\setminus \overline{B(x,\varepsilon)}$. Since $\varepsilon\in (0, \de(x))$, we conclude $\GDP f$ is finite a.e. on $D$. Now, take any such $x\in D$,  and note that by \eqref{eq:GDP sharp} we have $\GDP(x,y)\ge c(\de(x),d,D,\phi,\psi) V(\de(y))$, for $y\in D\setminus B(x,\de(x)/2)$. Hence
		\begin{align}\label{eq:1240}
			\infty>\GDP f(x)\gtrsim c(\de(x),d,D,\phi,\psi) \int_{D\setminus B(x,\de(x)/2)}V(\de(y))f(y)dy,
		\end{align}
		and by repeating the same reasoning as in \eqref{eq:1240}, it follows $f\in \LLL$. The claim for a general $f$ follows by $f=f^+-f^-$.
    \end{proof}

   Next, we present the sharp behaviour of Green potentials of a large class of functions, which is new even in the interpolated fractional Laplacian case. In classical fractional settings, this property is known in the form $G_D \de^\gamma\asymp \de^{\eta}$, see \cite{dhifli2012} and \cite[Theorem 3.4]{AGCV19}.

 \begin{thm}\label{t:bndry-G(U)}
		Let $U:(0,\infty)\to(0,\infty)$ be such that the following conditions hold:
		\begin{enumerate}[label=(\textbf{U\arabic*})]
			\item \label{U1}  
			$\int_{0}^1U(t)V(t)\, dt <\infty$;
			
			\item \label{U2} $U$ is almost non-increasing, i.e.~$\exists C>0$ such that $U(t)\le C U(s)$, $0<s\le t\le 1$;

			\item \label{U3} reverse doubling condition holds, i.e.~$\exists C>0$  
			such that $U(t)\le C U(2t)$, $t\in (0,1)$; 
			\item \label{U4} $U$ is bounded from
			above on $[c, \infty)$ for each $c >0$.
		\end{enumerate}
		Then $U(\de)\in\LLL$, and for all $x\in D$ it holds that
		\begin{equation}\label{eq:G(U) sharp bound}
    			\GDP\big(U(\de)\big)\asymp
    			\frac{1}{\de V(\de)\psi(V(\de)^{-2})}\int_0^{\de}U(t)V(t)\, dt+
    			V(\de)\left(1+\int_{\de}^{\diam D}
    			\frac{U(t)}{t V(t)\psi(V(t)^{-2})}\, dt\,\right).  
		\end{equation}
	\end{thm}
    The proof of Theorem \ref{t:bndry-G(U)} is provided in Appendix \ref{ap:s:G(U) approx} since it follows the lines of the proofs of \cite[Theorem 3.6]{Bio23cpaa} and \cite[Proposition 4.1]{BVW-na21}. 
    Whereas Theorem \ref{t:bndry-G(U)} is convenient for functions $U$ with singularities at 0, the case where $U$ decays at $0$ with a certain rate is covered by the following corollary, proved in much the same way as Theorem \ref{t:bndry-G(U)}.
    \begin{cor}\label{c:bndry-G(U)-decay}
     The estimate \eqref{eq:G(U) sharp bound} holds for all functions $U:(0,\infty)\to(0,\infty)$ satisfying \ref{U1} and \ref{U4} such that for every $c>0$ there exists $C=C(c)>1$ such that 
        \begin{align}\label{eq: U3'}
            \frac{1}{C}\, U(t)\le U(s)\le C\, U(t),\quad   {c}^{-1} t \le s \le c\,t, \ s,t\in (0,1].  \tag{\textbf{U3'}}
        \end{align}
    \end{cor}
 
    As mentioned before, the claims of Theorem \ref{t:bndry-G(U)} and Corollary \ref{c:bndry-G(U)-decay} for the power function $U$ have been already established in the case of the interpolated fractional Laplacian. 
    \begin{cor}\label{c:bndry behav of G(de^k)} Assume that $\phi(\lambda)=\lambda^{\beta/2}$ and $\psi(\lambda)=\lambda^{\alpha/2}$.	Let $\kappa\in(-1-\beta/2,\infty)$. Then
		\begin{align*}
			\GDP(\de^\kappa)\asymp\begin{cases}
				\de(x)^{\beta/2},&\kappa>\frac{\beta}{2}(1-\alpha),\\
				\de(x)^{\beta/2}\log\frac{\diam D}{\de(x)}&\kappa=\frac{\beta}{2}(1-\alpha),\\
				\de(x)^{\kappa+\frac{\alpha\beta}{2}}&\kappa<\frac{\beta}{2}(1-\alpha).
			\end{cases}
		\end{align*}
	\end{cor}

    \subsection{Homogeneous solutions to $\Lo u=f$}\label{ss:homo-sol}
    Our final goal of this section is to show that Green potentials are distributional solutions to the problem $\Lo u=f$. Recall that a function $u$ is a distributional solution to $\Lo u=f$ if
	\begin{align}\label{eq:dist-sol}
		\int_D u(x)\Lo \xi(x)dx=\int_D f(x)\xi(x)dx,\quad \xi \in C_c^\infty(D).
	\end{align}
    Therefore, to determine the appropriate sharp weighted $L^1$ space for distributional solutions, we first need to determine the sharp boundary behaviour of $\Lo \xi$, for $\xi \in C_c^\infty(D)$, which is governed by $J_D$. Indeed, by using the representation \eqref{eq:Lo point}, for $\xi \in C_c^\infty(D)$ and  $\de(x)\le \frac12\dist(D^c,\supp \xi)$ we get that
	$$ |\Lo \xi(x)|\lesssim \int_{\supp \xi}J_D(x,y)|\xi(y)|dy,$$
    where the equality holds if $\xi$ is also non-negative.
Although sharp bounds for $J_D(x,y)$ are in general not known, see Remark \ref{r:Lo pointwise}, we can apply Proposition \ref{p:Lo pointwise}, \eqref{eq:heat-sharp-small} and \eqref{eq:heat-sharp-big} to get that for $y\in \supp \xi$ and $x\in D$ such that $\de(x)\le \frac12\dist(D^c,\supp \xi)$
    \begin{align}\label{eq:dist-weight-rho}
		J_D(x,y)&\asymp \int_0^{ \diam D} \left(\frac{V(\de(x))}{\sqrt{t}}\wedge 1\right)\psi(1/t) dt+V(\de(x))\nonumber\\
    &\asymp \int_0^{V(\de)^2}\psi(1/t)dt+V(\de)\int_{V(\de)}^{\diam D}\psi(s^{-2})ds+V(\de)\nonumber\\
    &\asymp V(\de)^2\psi(V(\de)^{-2})+V(\de)\int_{V(\de)}^1\psi(s^{-2})ds\eqqcolon \rho(\de(x)).
	\end{align}
    Here, in the first line we used $\nu(t)\asymp \frac{\psi(1/t)}{t}$, $t<1$, see \cite[Corollary 2.5 \& Proposition 2.6]{KSV14-gubhp-spa}, and in the last line we used  $V(\de)\int_{V(\de)}^{ \diam D}\psi(s^{-2})ds\gtrsim V(\de)$. The constant of comparability above depends only on $d$, $D$, $\psi$, $\phi$, and $\supp \xi$.   
    Therefore, the appropriate choice of the function space for distributional solutions is
    \begin{align*}
        \calL\coloneqq L^1\big(D,\rho(\de(x))dx\big).
    \end{align*}
    \begin{rem}
        Under	suitable assumptions as in Remark \ref{r:Lo pointwise}, it is not hard to  get a more explicit behaviour of $\Lo \xi$. E.g., if
	$\psi(\lambda)=\INTFR$, then as in \cite[Example 8.5]{ksv2020boundary}
	\begin{align*}
		\rho(\de(x))\asymp  \begin{cases}
		\de(x)^{\beta/2}, &\alpha<1,\\
		\de(x)^{\beta/2}\log\frac{\diam D}{\de(x)},&\alpha=1,\\
		\de(x)^{\beta/2-\alpha\beta/2},&\alpha>1,
		\end{cases}
	\end{align*}
    \end{rem}
     By applying Corollary \ref{c:bndry-G(U)-decay}, we have the following cornerstone result regarding the spaces $\calL$ and $\LLL$ which ensures that Green potentials are in the appropriate space for distributional solutions.
    \begin{thm}\label{t:conti-GDP-L1}
        It holds that $\GDP\big(\rho(\de)\big)\asymp V(\de)$, and the map $\GDP:\LLL\to \calL$
        is continuous.
    \end{thm}
    \begin{proof}
    First, assume that $\GDP\big(\rho(\de)\big)\asymp V(\de)$. Then the continuity of $\GDP$ follows by Fubini's theorem and the symmetry of $G_D^\psi(x,y)$, since
    \begin{align*}
        \|\GDP f\|_{\calL}&\le \int_D \GDP|f|(x)\rho(\de(x))dx=\int_D |f|(x)\GDP\big(\rho(\de)\big)(x)dx\\
        &\asymp \int_D |f|(x)V(\de(x))dx=\|f\|_{\LLL}.
    \end{align*}
 To prove $\GDP\big(\rho(\de)\big)\asymp V(\de)$ we treat each term in $\rho(\de)$ separately. By \ref{as:WSC} for $\phi$ and $\psi$ we have that \eqref{eq: U3'} holds for $t\mapsto V(t)^2\psi(V(t)^{-2})$, see Lemma \ref{ap:l:rho}. Hence, the formula \eqref{eq:G(U) sharp bound} implies
    \begin{align*}
        \GDP\big(V(\de)^2\psi(V(\de)^{-2})\big)&\asymp \frac{1}{\de V(\de)\psi(V(\de)^{-2})}\int\limits_0^{\de}V(t)^3\psi(V(t)^{-2})\, dt\\
    			&\hspace{-1em}+
    			V(\de)+V(\de)\int\limits_{\de}^{\diam D}
    			\frac{V(t)}{t}\, dt\,.  
    \end{align*}
    Since $r\mapsto r^2\psi(r^{-2})$ is non-decreasing, and $V$ satisfies the weak scaling condition, the first term is bounded from above by $\frac{V(\de)}{\de}\int_0^{\de} V(t)dt\asymp  V(\de)^2\lesssim V(\de).$
     
    For the third term, we can again use the weak scaling property of $V$ to get $V(\de)\int_{\de}^{\diam D}
    			\frac{V(t)}{t}\, dt\lesssim  V(\de)\int_{0}^{\diam D}
    			\frac{V(t)}{t}\, dt\lesssim V(\de).$
    Therefore, $\GDP\big(V(\de)^2\psi(V(\de)^{-2})\big)\asymp V(\de)$.

    The second term of $\rho(\de)$ also satisfies \eqref{eq: U3'}, see Lemma \ref{ap:l:rho}.  Thus, we can apply Corollary \ref{c:bndry-G(U)-decay} to get
    \begin{align*}
        \GDP\left(V(\de)\int_{V(\de)}^1\psi(s^{-2})ds\right)&\asymp \frac{1}{\de(x)V(\de)\psi(V(\de)^{-2})}\int_0^{\de(x)}V(t)^2\int_{V(t)}^1\psi(s^{-2})ds\, dt\\
    			&\hspace{-1em}+
    			V(\de)+V(\de)\int_{\de}^{\diam D}
    			\frac{\int_{V(t)}^1\psi(s^{-2})ds}{t\psi(V(t)^{-2})}\, dt\,.  
    \end{align*}
    Note that by the weak scaling of $\psi$ we have
    \begin{align}\label{eq:1802}
        \int_{V(t)}^1\psi(s^{-2})ds\lesssim\begin{cases}
            V(t)^{2\ud{\psi}} \psi(V(t)^{-2}),& \ud{\psi}<\frac12,\\
            V(t)\psi(V(t)^{-2})\log\big(1/t\big),& \ud{\psi}=\frac12,\\
            V(t)\psi(V(t)^{-2}),& \ud{\psi}\ge \frac12.
        \end{cases} \lesssim V(t)^{\eta}\psi(V(t)^{-2})
    \end{align}
    for some $\eta\in (0,1)$. Hence, it follows that
    \begin{align*}
       \GDP\left(V(\de)\int_{V(\de)}^1\psi(s^{-2})ds\right)&\lesssim V(\de)^{1+\eta}+
    			V(\de)\left(1+\int_{\de}^{\diam D}
    			\frac{V(t)^{\eta}}{t}\, dt\,\right)\lesssim V(\de),
    \end{align*}
    where for the first term we used that $r\mapsto r^2\psi(^{-2})$ is non-decreasing, and that $V$ is increasing, and for the third we note that $\int_{0}^{\diam D}\frac{V(t)^{\eta}}{t}\, dt<\infty$ since $V(t)\lesssim t^{\ud{\phi}}$ by the weak scaling at zero and $\eta \ud{\phi}<1$.
    This finishes the proof of $\GDP\rho(\de)\asymp V(\de)$.
    \end{proof}

	\begin{thm}\label{t:GDPf dist-sol}
		For $f\in \LLL$, the function $u=\GDP f$ solves $\Lo u=f$ in $D$
		distributionally, i.e. \eqref{eq:dist-sol} holds.
	\end{thm}
	\begin{proof}
		Let $f\in \LLL$ and  $\xi\in C_c^\infty(D)$. By Theorem \ref{t:conti-GDP-L1} it follows that $\GDP f\in \calL$, and by applying \eqref{eq:dist-weight-rho} one can easily show that $|\Lo \xi|\le C(\phi,\psi,\xi) \rho(\de)$.  Moreover, by Remark \ref{r:infty-gener}, we have $\xi \in \DD(\Lo)\subset L^2(D)$ and $\Lo \xi \in L^2(D)$, so by Proposition \ref{p:Lo inverse L2(D)} it holds that $\GDP\Lo \xi=\xi$.
		
		Fubini's theorem, and Propositions \ref{p:Lo inverse L2(D)}
		and \ref{p:composition-Lo} yield
		\begin{align*}
			\int_D \GDP f(x)\Lo \xi(x)dx&=\int_D \int_D \GDP(x,y)\Lo \xi (x)
			dxf(y)dy\\
            & =\int_D G_D^\psi \big(\Lo \xi\big)(y)f(y)dy\\
			&=\int_D  \xi(y)f(y)dy,
		\end{align*}
		which proves the claim.
	\end{proof}
	We finish with a regularity property of the Green potentials, which follows from joint continuity of $\GDP$ from Lemma \ref{l:GDP cont} and an analogous proof as in \cite[Proposition 2.16 \& Remark 2.17]{Bio23cpaa}.
        \begin{prop}\label{p:uniform}
            Let $f\in L_{loc}^\infty(D)\cap \LLL$. Then 
            \begin{align*}
                \lim_{z\to x\in D}\int_D|\GDP(z,y)-\GDP(x,y)||f(y)|dy=0,
            \end{align*}
            uniformly on compact subsets of $D$. In particular, $\GDP f \in C(D)$.
        \end{prop}
	
	\section{Poisson kernel and harmonic functions}\label{s:Poisson harmonic}
The Poisson kernel of $\Lo$ is defined as a generalised normal derivative of the
	Green function:
	\begin{align*}
		\PDS(x,z) \coloneqq -\partial_V\GDP(x,z)=\lim_{D\ni y\to
			z}\frac{\GDP(x,y)}{V(\de(y))},\quad x\in D,\, z\in \partial D.
	\end{align*}
	\begin{prop}\label{p:Poisson kernel}
		The function $\PDS$ is well defined, jointly continuous on $D\times	\partial D$, and 
		\begin{align}\label{eq:Poisson kernel sharp}
			\PDS(x,z)\asymp
			\frac{V(\de(x))}{V(|x-z|)^2|x-z|^{d}\psi(V(|x-z|)^{-2})},\quad x\in D,\, z\in
			\partial D.
		\end{align}
		Further,
		\begin{align}\label{eq:Green Poisson identity}
			\int_D \GDPs(x,\xi)\PDS(\xi,z)d\xi=K_D(x,z),
		\end{align}
		where $K_D(x,z)=\lim\limits_{D\ni y\to z}\frac{G_D(x,y)}{V(\de(y))}$ is the modified Martin kernel in $D$ of the L\'evy process $X$.
	\end{prop}
    
\begin{proof}
		The proof follows the lines of \cite[Proposition 2.19]{Bio23cpaa}, with the major difference in the calculations arising from the Green function estimate \eqref{eq:GDP sharp}. By the representation \eqref{eq:GDP_representation} of $\GDP$, we have that
		\begin{align}\label{eq:poisson_dominated}
			 \lim\limits_{D\ni y\to z} \frac{\GDP(x,y)}{V(\de(y))}= \lim\limits_{D\ni y\to z}\int_0^\infty
			\frac{p_D(t,x,y)}{V(\de(y))}\uu(t)dt =\int_0^\infty
			\lim\limits_{D\ni y\to z}\frac{p_D(t,x,y)}{V(\de(y))}\uu(t)dt,
		\end{align}
        where the application of the dominated convergence theorem in \eqref{eq:poisson_dominated} is justified by the estimates \eqref{eq:heat-sharp-small} and \eqref{eq:heat-sharp-big}, integrability of  $\uu$ near 0, and boundedness of $\uu$ away from 0.
        Since the limit $\partial_V p(t,x,z):=\lim\limits_{D\ni y\to z}{p_D(t,x,y)}/{V(\de(y))}$ exists and is jointly continuous in $(0,\infty)\times D\times\partial D $ by Lemma \ref{ap:l:joint p_D}, we obtain the existence and joint continuity in $D\times\partial D$  of the Poisson kernel.  Furthermore, the sharp bounds \eqref{eq:Poisson kernel sharp} for $\PDS$ follow  easily from \eqref{eq:GDP sharp}.  For \eqref{eq:Green Poisson identity}, first note that $K_D(x,z)$ is well defined, see \cite[Subsection 4.6]{BVW-na21} and Remark \ref{r:Bio21jmaa}. By Lemma \ref{l:compos-of-GDs},
		\begin{align}\label{eq:1055}
			 \lim\limits_{D\ni y\to z}\frac{1}{V(\delta_D(y))}\int_D \GDPs(x,\xi)\GDP(\xi,y)d\xi=- \partial_V G_D(x,z)=K_D(x,z).
		\end{align}
		 In order to justify the application of the dominated convergence theorem to the left-hand side of \eqref{eq:1055}, which proves \eqref{eq:Green Poisson identity},  in what follows we use analogous calculations as in \cite[Proposition 2.19]{Bio23cpaa}, accounting for necessary changes based on the Green function estimates \eqref{eq:GDP sharp}.  Fix $x\in D$ and $z\in \partial D$, and take $\varepsilon>0$ such that $\de(x)>3\varepsilon$ and $y\in D$ such that $|z-y|<\varepsilon/2$.  First, for $\xi \in D\setminus B(z,\varepsilon)$ we have $|\xi-y|\asymp \varepsilon$, and therefore
		\begin{align}\label{eq:upper1}
			\frac{\GDPs(x,\xi)\GDP(\xi,y)}{V(\de(y))}\lesssim \frac{\GDPs(x,\xi)V(\de(\xi))}{V(|\xi-y|)^2|\xi -y|^d\psi(V(|\xi-y|)^{-2})}\lesssim \GDPs(x,\xi)V(\de(\xi)),
		\end{align}
	where the right-hand side is an integrable function of $\xi$ by Lemma \ref{l:Green invariant}. Next, note that for $\xi \in B(z,\varepsilon)\cap D$, $|\xi-x|\asymp \varepsilon$, so by \eqref{eq:GDP sharp} we have $\GDPs(x,\xi)\lesssim V(\de(\xi))$. This, together with \eqref{eq:GDP sharp} and \ref{as:WSC}, implies that for $\xi\in B(y,\de(y)/2)$ we have
		\begin{align}\label{eq:upper2}
			\frac{\GDPs(x,\xi)\GDP(\xi,y)}{V(\de(y))}&\lesssim \frac{V(\de(\xi))}{V(\de(y))}\frac{1}{|\xi-y|^d\psi(V(|\xi-y|)^{-2})}\lesssim \frac{1}{|\xi-y|^d\psi(V(|\xi-y|)^{-2})}.
		\end{align}
		Similarly, for $\xi\in D\cap B(z,\varepsilon)\setminus B(y,\de(y)/2)$ note that $\de(\xi)\le 3|\xi-y|\le 2\varepsilon$, so
		\begin{align}\label{eq:upper3}
			\frac{\GDPs(x,\xi)\GDP(\xi,y)}{V(\de(y))}&\lesssim\frac{V(\de(\xi))}{V(\de(y))}\frac{V(\de(y)) }{V(|\xi-y|) |\xi-y|^d\psi(V(|\xi-y|)^{-2})}\lesssim\frac{1 }{|\xi-y|^d\psi(V(|\xi-y|)^{-2})}.
		\end{align}
	The proof is finished by noting that the constants in \eqref{eq:upper1}, \eqref{eq:upper2} and \eqref{eq:upper3} are independent of $y$ and that the right-hand sides of the latter two inequalities are integrable functions of $\xi$ in $D$ by \ref{as:WSC} for $\psi$ and \eqref{e:wsc-V}.
	\end{proof}
The Poisson potential $\PDS \zeta$ of a measure $\zeta$ on $\partial D$ is defined by
		\begin{align*}
			\PDS \zeta(x)\coloneqq \int_{\partial D}\PDS(x,z)\zeta(dz),\quad x\in D.
		\end{align*}
From the sharp bound \eqref{eq:Poisson kernel sharp} we easily see that $\PDS\zeta$ is finite if and only if $\zeta$ is finite.

	\begin{prop}\label{p:poisson bndry and finiteness}
	For the $(d-1)$--Hausdorff measure $\sigma$ on $\partial D$, it follows that
		\begin{align}\label{eq:Poisson potential sharp}
			\PDS \sigma(x)\asymp \frac{1}{\de(x) V(\de(x)) \psi(V(\de(x))^{-2})},\quad x\in D.
		\end{align}
	\end{prop}
	\begin{proof} 
    	Let $x\in D$. Since $r\mapsto r^2\psi(r^{-2})$ and $V$ are non-decreasing, and $\delta_D(x)\le |x-z|$ for $z\in\partial D$,  \eqref{eq:Poisson kernel sharp} implies that
		\begin{align*}
			&\PDS\sigma(x)\lesssim \int_{\partial D}\frac{V(\de(x))}{V(|x-z|)^2|x-z|^d\psi(V(|x-z|)^{-2})}\sigma(dz)\lesssim\frac{1}{V(\de(x))\psi(V(\de(x))^{-2})}\int_{\partial D} \frac{\sigma(dz)}{|x-z|^d}.
		\end{align*}
		Set $\Gamma=\{z\in\partial D: |x-z|\le 2\de(x)\}$. The same arguments as above imply that
		\begin{align*}
			&\PDS\sigma(x)\gtrsim \int_{\Gamma}\frac{V(\de(x))}{V(|x-z|)^2|x-z|^d\psi(V(|x-z|)^{-2})}\sigma(dz)\gtrsim\frac{1}{V(\de(x))\psi(V(\de(x))^{-2})}\int_{\Gamma} \frac{\sigma(dz)}{|x-z|^d}.
		\end{align*}
		Now \eqref{eq:Poisson potential sharp} follows from  $\int_{\Gamma}\frac{\sigma(dz)}{|x-z|^d}\asymp\int_{D} \frac{\sigma(dz)}{|x-z|^d}\asymp \de(x)^{-1}$, as proved in \cite[Lemma A.2]{Bio23cpaa}.
	\end{proof}

        \begin{cor}\label{c:PDS-inter-sharp}
            If $\Lo =\INTFR$, it holds that
            \begin{align*}
            	\PDS(x,z)&\asymp \frac{\de(x)^{\beta/2}}{|x-z|^{d+\beta-\frac{\alpha\beta}{2}}},\quad x\in D,\, z\in \partial D,\\
            	\PDS\sigma(x)&\asymp \de(x)^{-1-\frac\beta2+\frac{\beta\alpha}2},\quad x\in D.
            \end{align*} 
        \end{cor}
	
	\begin{lem}\label{l:integrability of Poisson potentials}
		If $\zeta$ is a finite measure on $\partial D$, then $\PDS\zeta \in \LLL \cap \calL$ and 
		\begin{align*}
			\|\PDS\zeta\|_{\LLL}+\|\PDS\zeta\|_{\calL}\le C |\zeta|(\partial D),
		\end{align*}
		for some constant $C=C(d,D,\psi,\phi)>0$.  Moreover, $\PDS\zeta \in C(D)$.
	\end{lem}
	\begin{proof}
        By applying \eqref{eq:Poisson kernel sharp}, note that $\|\PDS\zeta\|_{\LLL}$ is dominated from above by 
        \begin{align*}
            \int_{\partial D}\int_D \PDS(x,z)V(\de(x))dx|\zeta|(dz)
        	 \lesssim\int_{\partial D}\int_{D} \frac{dx}{|x-z|^d\psi(V(|x-z|)^{-2})}|\zeta|(dz)\lesssim |\zeta|(\partial D),
        \end{align*}
        since the function $r\mapsto (r\psi(V(r)^{-2}))^{-1}$ is integrable on $(0,\text{diam}(D))$, by \ref{as:WSC} for $\psi$ and \eqref{e:wsc-V}.
        Similarly, $\|\PDS\zeta\|_{\calL}$ is bounded from above by
        $$\int_D \PDS|\zeta|(x)\rho(\de(x))dx=\int_{\partial D}\left(\int_{D\cap B(z,1)^c}+\int_{D\cap B(z,1)}\right)\PDS(x,z)\rho(\de(x))dx|\zeta|(dz)\eqqcolon I_1 + I_2.
        $$
        Note that $\rho(\de)$ is bounded and vanishes at the boundary, and that on $D\cap B(z,1)^c$ we have $\PDS(x,z)\lesssim V(\de(x))$, hence, $I_1\lesssim |\zeta|(\partial D)$.  By \eqref{eq:Poisson kernel sharp} and \eqref{eq:dist-weight-rho}, the inner integral in $I_2$ is split into two parts:
        \begin{align*}
            &\int_{D\cap B(z,1)}\PDS(x,z)\rho(\de(x))dx\asymp
            \int_{D\cap B(z,1)}\frac{V(\de(x))^3 \psi(V(\de(x))^{-2})dx}{V(|x-z|)^2|x-z|^d\psi(V(|x-z|)^{-2})}\\
            &\hspace{6em}+\int_{D\cap B(z,1)}\frac{V(\de(x))^2\int_{V(\de(x))}^1 \psi(s^{-2})ds}{V(|x-z|)^2|x-z|^d\psi(V(|x-z|)^{-2})}dx\eqqcolon J_1+J_2.
        \end{align*}
        Since $\de(x)\le |x-z|$ and $r\mapsto r^2\psi(r^{-2})$ and $V$ are non-decreasing, \eqref{e:wsc-V} implies that
        $J_1\lesssim \int_{B(z,1)}\frac{V(|x-z|)}{|x-z|^d}dx\lesssim 1.$
        Further, by first applying the inequality \eqref{eq:1802}, and then using the analogous reasoning as for the estimate of $J_1$, we also get that
        $J_2\lesssim \int_{B(z,1)}\frac{V(|x-z|)^{\eta}}{|x-z|^d}dx\lesssim 1.$
        
        Hence $\PDS\zeta\in \calL$  and $\|\PDS\zeta\|_{\calL}\le C(d,\psi,\phi,D) |\zeta|(\partial D)$. The continuity of $\PDS\zeta$ follows by the dominated convergence theorem and Proposition \ref{p:Poisson kernel}.
        \end{proof}
	
		It is known that Poisson kernels are, in general, connected to harmonic functions via their integral representation. Our next goal is to prove this relation, as well as to relate harmonic functions with respect to $\Lo$ to the harmonic functions with respect to $\Lbase$.  We note that we consider harmonic functions in the distributional and probabilistic sense.
        
	\begin{defn}\label{def:harmonic}
		The function $h\in \calL$ is harmonic in $D$ with respect to $\Lo$ if $\Lo h=0$
		in $D$ in the distributional sense.
	\end{defn}
	\begin{prop}\label{p:harmonic connect}
		The function $h\in \calL$ is harmonic in $D$ with respect to $\Lo$ if and only if
		$\GDPs h$ is harmonic in $D$ with respect to $\Lbase$. Further,
		for every finite measure $\zeta$ on $\partial D$, the function $\PDS \zeta$ is
		harmonic in $D$ with respect to $\Lo$.
	\end{prop}
	\begin{proof}
		For $h\in \calL$ and $\xi\in C_c^\infty(D)$, we have by Propositions \ref{p:Lo inverse L2(D)} and  \ref{p:composition-Lo} that 
		\begin{align*}
			&\int_D h(x) \Lo \xi(x)dx=
			\int_D h(x) \GDPs  \big(-\Lbase\xi\big)(x)dx=\int_D \GDPs h(x) \big(-\Lbase\xi\big)(x)dx,
		\end{align*}
		where in the second equality  we used Fubini's theorem. This is justified by noting that $\Lbase \xi= L\xi \in L^\infty(D)$, by \cite[Eq. (6)]{Bio21jmaa}, and that $\GDPs h\in L^1(D)$. The latter is a consequence of the estimate $\GDPs \1 \asymp \rho(\de)$.
		Indeed, by Theorem \ref{t:bndry-G(U)} for $\psi^*$,
		\begin{align*}
			\GDPs \1&\asymp 
			\frac{1}{\de(x)V(\de)\psi^*(V(\de)^{-2})}\!\!\!\int\limits_0^{\de(x)}\!\!\!V(t)\, dt+V(\de(x))+V(\de(x))\!\!\!\int\limits_{\de(x)}^{\diam D}
			\!\!\!\frac{dt}{t V(t)\psi^*(V(t)^{-2})},
		\end{align*}
        	which is by \eqref{e:wsc-V} and $V'(t)\asymp V(t)/t$, $t<1$, comparable to $\rho(\de)$.
		
		For the second claim, take a non-negative finite measure $\zeta$ on $\partial D$. From Proposition \ref{p:Poisson kernel}, it follows that $\GDPs \PDS \zeta = K_D\zeta$.  Therefore, by the first part of the proposition, it suffices to show that $K_D\zeta$ is $\Lbase$-harmonic. This follows from the Martin representation of the non-negative harmonic functions with respect to $X^D$, proved in \cite{Bio21jmaa} (see also Remark \ref{r:Bio21jmaa}). Indeed, by \cite[Eq. (4.24)]{BVW-na21}, we have $K_D\zeta=M_D\wt\zeta$, where $M_D(x,z)=\lim_{y\to z}G_D(x,y)/G_D(x_0,y)$, for $x\in D$, $z\in \partial D$, and some fixed $x_0\in D$, is the classical Martin kernel of $X^D$, while $\wt\zeta(d z)=K_D(x_0,z)\zeta(dz)$. Further, $M_D\wt\zeta$ is $\Lbase$-harmonic by \cite[Proposition 5.11 \& Theorem 3.16]{Bio21jmaa}, hence $K_D\zeta$ is $\Lbase$-harmonic. For a general finite measure $\zeta$, we use the decomposition $\zeta=\zeta^+-\zeta^-$.
	\end{proof}

	Now we prove the Martin integral representation of non-negative $\Lo$-harmonic functions.
	\begin{prop}\label{p:harm int rep}
		A non-negative function $h\in \calL$ is  $\Lo$-harmonic if and only if there exists a finite measure $\zeta$ on $\partial D$ such that 
		\begin{align*}
			h(x)=\PDS \zeta(x), \quad \text{for a.e. $x\in D$}.
		\end{align*}
	\end{prop}
	\begin{proof} Note that one implication is shown in Proposition \ref{p:harmonic connect}. Let $h\in \calL$ be a non-negative $\Lo$-harmonic function. By Proposition \ref{p:harmonic connect} it follows that the non-negative function $\GDPs h$ is $\Lbase$-harmonic. By employing the representation of non-negative $\Lbase$-harmonic functions, see \cite[Proposition 5.11 \& Theorem 3.16]{Bio21jmaa}, \cite[Eq. (4.24)]{BVW-na21}, and Remark \ref{r:Bio21jmaa}, there is a finite measure $\zeta$ on $\partial D$ such that $\GDPs h=K_D\zeta$ a.e.~in $D$. By applying \eqref{eq:Green Poisson identity} and Fubini's theorem to the right-hand side of this equality, we obtain that $\GDPs h=\GDPs\PDS\zeta$ a.e.~in $D$. By Theorem \ref{t:GDPf dist-sol}., it follows that $h=\PDS \zeta$ a.e.~in $D$.
	\end{proof}
	
	Next, we turn to the probabilistic definition of a harmonic function and relate it to harmonic functions in Definition \ref{def:harmonic}. A function $h$ is $Y$-harmonic in $D$ if for all $U\subsub D$ it holds that
	\begin{align*}
		h(x)=\ex_x[h(Y_{\tau_U})],\quad x\in U,
	\end{align*}
	where $\tau_U=\inf\{t>0:Y_t\notin U\}$ and we assume the expectation in the definition is implicitly finite. 
    \begin{rem}
       For a regular enough domain $U$, the distribution of $Y_{\tau_U}$ is determined by the Poisson kernel of the process $Y$ killed upon exiting the set $U$, as a consequence of the L\'evy system formula, see e.g. \cite[p. 1241]{ksv_minimal2016} or \cite{ikeda-watanabe}. Using this kind of representation, one can see that $Y$-harmonic functions are in $\calL$, see the approach in \cite[Eqs. (3.22) \& (3.23)]{ksv_minimal2016} and use that therein $U^{D,B}(x,z)\le \GDP(x,z)$ and $J^D(z,y)=J_D(z,y)\lesssim \rho(\de(y))$, for $z$ away from the boundary.
    \end{rem}

	\begin{thm}\label{t:1627}
		The non-negative function $h\in \calL$ is harmonic in $D$ with respect to $\Lo$
		if and only if it satisfies the mean-value property with respect to $Y$ in $D$ (after a modification on a Lebesgue null set).
	\end{thm}
	
	\begin{proof}
	The proof is based on \cite{SongVondracek2006potenSpecial}, where the underlying process $X$ is an isotropic stable process. In \cite[Theorem 3.6]{SongVondracek2006potenSpecial} authors show  that a non-negative function $h$ has the mean value property with respect to $Y$ if and only if $\GDPs h$ has a mean value property with respect to $X^D$. By carefully examining the proofs in \cite{SongVondracek2006potenSpecial}, we can see that these claims hold even in the more general setting of this paper, where $X$ is a more general subordinate Brownian motion. As stated at the beginning of \cite[Section 3]{SongVondracek2006potenSpecial}, in order to generalize \cite[Theorem 3.6]{SongVondracek2006potenSpecial}, one needs to obtain the following result: if $h$ is a non-negative function with the mean value property with respect to $X^D$, then $h$ and $P^D_t h$ are continuous in $D$. In our setting, the smoothness of $h$ is obtained in \cite[Theorem 3.12]{Bio21jmaa}, see also Remark \ref{r:Bio21jmaa}, while the continuity of $P^D_t h$ follows by the dominated convergence theorem from the continuity of  $p_D$, see Lemma \ref{ap:l:joint p_D}, and the fact that $h\in L^1(D)$, see \cite[Lemma 3.8]{Bio21jmaa}.

		That being said,  we can establish that  a non-negative function $h\in \calL$  has the mean value property with respect to $Y$ if and only if $\GDPs h$ has the mean value property with respect to $X^D$, which by \cite[Proposition 5.11]{Bio21jmaa} and \cite[Eq. (4.24)]{BVW-na21} holds if and only if $\GDPs h=K_D\zeta$ with some finite measure $\zeta$ on $\partial D$.  Thus, by employing Theorem \ref{t:GDPf dist-sol}, a non-negative function $h\in \calL$ has a mean value property with respect to $Y$ if and only if $h=\PDS \zeta$, which by Proposition \ref{p:harm int rep} holds if and only if $h$ is $\Lo$-harmonic.
	\end{proof}
	
	The probabilistic representation of harmonic functions will be of great importance in Section \ref{s:semilinear} dealing with semilinear problems. For example, in Theorem \ref{t:non-positive semilinear problem} we will apply an approximation of harmonic functions by a sequence of Green potentials, a construction that is enabled by probabilistic techniques for harmonic functions.  Lastly, we note a consequence of the existence of the Poisson kernel $\PDS$ on the Martin boundary of the process $Y$.

    \begin{rem}\label{r:martin}
The Poisson kernel $\PDS$ can be viewed as a modified Martin kernel of the process $Y$, where the Martin kernel of $Y$ is given by
	\begin{align*}
		M^{x_0}_D(x,z)=\lim_{ D\ni  y\to z}\frac{\GDP(x,y)}{\GDP(x_0,y)}=\frac{\PDS(x,z)}{\PDS(x_0,z)},\quad x\in D,\,z\in
		\partial D,
	\end{align*}
	for some fixed $x_0\in D$. In general, the limit in this definition is taken in the abstract Martin topology, see \cite{KW_martin}. However, a generally abstract Martin topology and the Martin boundary\footnote{For details we refer to \cite[p. 560 and p. 546]{ksv_RMI18}, see also \cite{KW_martin}.} are in this case the usual Euclidean topology and boundary $\partial D$. The latter can be shown by an appropriate modification of \cite[Theorem 4.3]{SongVondracek2006potenSpecial}. Note that in our setting, the process $X^D$ plays the role of the killed isotropic stable process in \cite{SongVondracek2006potenSpecial}. As in the proof of Theorem \ref{t:1627}, properties of $X^D$ needed to obtain \cite[Theorem 4.3]{SongVondracek2006potenSpecial}  follow from \cite{Bio21jmaa}, see also Remark \ref{r:Bio21jmaa}.
    
    Furthermore, note that the usual way to obtain the Martin kernel is by using the boundary Harnack principle, which in our setting does not generally have to hold, see \cite[Section 9]{ksv2020boundary}. In a similar setting to ours, the existence of $M_D^{x_0}$ in the Euclidean topology was shown in \cite[Section 4]{SongVondracek2006potenSpecial}.
    \end{rem}
    
	\section{Boundary behaviour of potential integrals}\label{s:boundary}
	In this section, we study the boundary behaviour of Poisson and Green potentials relative to $\PDS\sigma$, which will lead us to the description of the nonhomogeneous boundary condition for $\Lo$.
	
	First, we study the behaviour of Poisson potentials, where we prove that the explosion rate of the non-negative $\Lo$-harmonic (i.e. $Y$-harmonic) functions is described by the reference function $\PDS\sigma$. 
	
	\begin{prop}\label{p:bnd-bhv-PDSzeta}
		Let $\zeta\in L^1(\partial D)$. It holds that 
		\begin{align*}
			\frac1t\int_{\{\de(x)\le t\}}
			\frac{\PDS\zeta(x)}{\PDS\sigma(x)}\varphi(x)dx
			\xrightarrow{t\downarrow 0} \int_{\partial D}\varphi(y)\zeta(y)d\sigma(y),\quad \varphi\in C(\overline D).
		\end{align*}
		Moreover, for a continuity point $z\in \partial D$ of $\zeta$, it holds that
		\begin{align}\label{eq:1134b}
			\lim_{D\ni x\to z\in\partial D}\frac{\PDS\zeta(x)}{\PDS\sigma(x)}=\zeta(z),
		\end{align}
		with the convergence in \eqref{eq:1134b} being uniform on $\partial D$ if $\zeta\in C(\partial D)$.
	\end{prop}
	\begin{proof}
		Let us first prove the second claim, which follows the line of \cite[Proposition 3.4]{Bio23cpaa}.  By \eqref{eq:Poisson kernel sharp} and \eqref{eq:Poisson potential sharp}, for all $x\in D$ and $y\in \partial D$ we have
		\begin{align}\label{eq:1107}
			\frac{\PDS(x,y)}{\PDS\sigma(x)}\lesssim
			\frac{\de(x)}{|x-y|^{d}}
			\frac{V(\de(x))^2 \psi(V(\de(x))^{-2}}{V(|x- y |)^2 \psi(V(|x-y |)^{-2}}\le
			\frac{\de(x)}{|x-y|^{d}}.
		\end{align}
		Take now a continuity point $z\in \partial D$ of $\zeta$, $\varepsilon>0$ and choose $\eta>0$ such that if $|y-z|<\eta$, then $|\zeta(y)-\zeta(z)|\le \varepsilon$. Define $\Gamma_z=\{y\in \partial D:|y-z|<\eta\}$, so for $|x-z|<\eta/2$, by using \eqref{eq:1107} and \cite[Lemma A.2]{Bio23cpaa}, we get
		\begin{align*}
			\left|\frac{\PDS\zeta(x)}{\PDS\sigma(x)}-\zeta(z)\right|&\lesssim \de(x)\int_{\Gamma_z}\frac{|\zeta(y)-\zeta(z)|}{|y- x|^d}\sigma(dy)+\de(x)\int_{\partial D\setminus \Gamma_z}\frac{|\zeta(y)-\zeta(z)|}{|y- x|^d}\sigma(dy)\\
			&\lesssim \varepsilon + \de(x)\left(\frac{\eta}{2}\right)^{-d}\|\zeta\|_{L^1(\partial D)},
		\end{align*}
		where the constant of comparability does not depend on $z$. Hence, we obtain \eqref{eq:1134b}. The uniformness, if $\zeta$ is continuous, follows by the uniform continuity of $\zeta$ since $\partial D$ is compact.
		
		The proof of the first claim follows exactly as the proof of \cite[Theorem 26]{AbaDup-Nonhomog2017} for the spectral fractional Laplacian, see also \cite[Proposition 3.5]{Bio23cpaa} for additional details.
	\end{proof}

	In the next proposition we deal with boundary behaviour of Green potentials with respect to the $\PDS\sigma$. As expected, $\PDS\sigma$ annihilates Green potentials but only in the weak sense. For a pointwise annihilation, we must assume a special form of functions.
	\begin{prop}\label{p:boundary operator GD lambda}
	    Let $\lambda$ be a measure on $D$ such that $\int_D V(\de(x))|\lambda|(dx)<\infty$. Then it holds
        \begin{align*}
            \lim_{t\downarrow 0}\frac{1}{t}\int_{\{\de(x)<t\}}\frac{\GDP \lambda(x)}{\PDS\sigma(x)}\varphi(x)dx=0,\quad \varphi\in C(\overline D).
        \end{align*}
        If $\lambda$ has a density, i.e. $\lambda(dx)=\lambda(x)dx$, such that $\lambda(x)\lesssim U(\de(x))$, where $U$ satisfies \ref{U1}--\ref{U4}, then
        \begin{align}\label{e:1040}
        	\lim_{D\ni x\to z\in\partial D}\frac{\GDP
        		\lambda(x)}{\PDS\sigma(x)}=0.
        \end{align}
	\end{prop}
    \begin{proof}
        For the first claim, by Fubini's theorem we get that
        \begin{align*}
        	\int_{\{\de(x)<t\}}\frac{\GDP \lambda(x)}{\PDS\sigma(x)}\varphi(x)dx=\int_D\int_{\{\de(x)<t\}}\frac{\GDP(x,y)}{\PDS\sigma(x)}\varphi(x)dx\,\lambda(dy).
        \end{align*}

    For the inner integral we will show that
        \begin{align}\label{eq:1442}
            \int_{\{\de(x)<t\}}\frac{\GDP(x,y)}{\PDS\sigma(x)}dx\le f(t,y),
        \end{align}
        where $f(t,y)\lesssim t V(\de(y))$ and $f(t,y)/t\to 0$ as $t\to 0$ for every fixed $y\in D$.
        
        To this end, for $\de(y)<\frac{t}{2}$, we first note that $U(t)=t\,V(t)\psi(V(t)^{-2})$ satisfies \ref{U1}, \eqref{eq: U3'}, and \ref{U4}, and that $U(\de)\asymp 1/\PDS\sigma$. By repeating the steps in Lemma \ref{apx:l:GDFI close-to-bdry} for $\eta=t$, we get that for $\de(y)<\frac{t}{2}$
        \[
        \int_{\{\de(x)<t\}}\frac{\GDP(x,y)}{\PDS\sigma(x)}dx\le Ct V(\de(y))=:f(t,y).        
        \]

        For $\de(y)\ge \frac{t}2$, we repeat the steps from the proof of \cite[Lemma A.5]{Bio23cpaa}. We split the integral in \eqref{eq:1442} into four terms, defined below,
        \begin{align*}
        	\int_{\{\de(x)<t\}}\frac{\GDP(x,y)}{\PDS\sigma(x)}dx=J_1+J_2+J_3+J_4.
        \end{align*}
        By \eqref{eq:GDP sharp} and \eqref{eq:Poisson potential sharp} we have that
        \begin{align*}
        	J_1&\coloneqq \int_{\{\de(x)<t/4\}}\frac{\GDP(x,y)}{\PDS\sigma(x)}dx\lesssim V(\de(y))\int_0^{t/4}\frac{h}{\de(y)}dh
        	\le  C_1\frac{V(\de(y))}{\de(y)}\,t^2\\
        \end{align*}
        and
        \begin{align*}
         	J_2&\coloneqq \int_{\{t/4<\de(x)<t\}\cap B(y,t/4)}\frac{\GDP(x,y)}{\PDS\sigma(x)}dx\lesssim \int_{B(y,t/4)}\frac{t V(t) \psi(V(t)^{-2}) }{|x-y|^d \psi(V(|x-y|)^{-2})}dx
        	\le C_2 t V(t).
        \end{align*}
        Further, define $D_3\coloneqq \{t/4<\de(x)<t\}\cap B(y,t/4)^c\cap B(y,r_0)$, for some fixed and small enough $r_0>0$ dependent on the characteristics of the set $D$, see Appendix \ref{ap:s:G(U) approx}. We have
		\begin{align*}
			J_3 &\coloneqq \int_{D_3}\frac{\GDP(x,y)}{\PDS\sigma(x)}dx\asymp \int_{D_3} \frac{ V(\de(x))V(\de(y))t V(t)^2 \psi(V(t)^{-2})}{V(|x-y|)^2|x-y|^d\psi(V(|x-y|)^{-2})}dx\\
			&\le C_3  t\,V(\de(y))\int_{D_3}|x-y|^{-d}dx.
		\end{align*}
		Finally, define $D_4\coloneqq \{t/4<\de(x)<t\}\cap B(x,r_0)^c$ so we have $J_4\coloneqq \int_{D_4}\frac{\GDP(x,y)}{\PDS\sigma(x)}dx\le C_4 V(\de(y))t\,V(t)^2\psi(V(t)^{-2}).$
        
       For $\de(y)\ge \frac{t}2$ set
       \begin{align*}
       f(t,y)&:=C_1\frac{V(\de(y))}{\de(y)}\,t^2+C_2 t V(t)+C_3  t\,V(\de(y))\int_{D_3}|x-y|^{-d}dx\\
       &\hspace{5em}+C_4 V(\de(y))t\,V(t)^2\psi(V(t)^{-2})
       \end{align*}
       and note that by \cite[Eq. (A.30)]{Bio23cpaa}, $\int_{D_3}|x-y|^{-d}dx\lesssim \int_{t/4}^{r_0}t\, r^{-2}dr\lesssim 1$ so by the dominated convergence theorem we have $\lim\limits_{t\to 0}\int_{D_3}|x-y|^{-d}dx=0$. Furthermore, since $\lim\limits_{r\to 0}r\psi(r^{-1})=0$ we have that $f(t,y)\lesssim tV(\de(y))$ and
       $\lim_{t\to 0}\frac{f(t,y)}{t}=0$ for all $y\in D$.
                
        For the second claim, assume that $\lambda(x)\lesssim U(\de(x))$, where $U$ satisfies \ref{U1}--\ref{U4}. It is enough to show that 
        \begin{align*}
        	\lim_{D\ni x\to z\in\partial D}\frac{\GDP
        		\big(U(\de)\big)(x)}{\PDS\sigma(x)}=0.
        \end{align*}
        We use the sharp bounds for $\GDP U(\de)$ obtained in \eqref{eq:G(U) sharp bound} and note that $\PDS\sigma$ annihilates the first and the second term of  \eqref{eq:G(U) sharp bound}. In the last term, on $\{t\ge \de(x)\}$ we have $t V(t)\psi(V(t)^{-2})\gtrsim \de(x)V(\de(x))\psi(V(\de(x))^{-2})$ and  $U(t)V(\de(x))\le U(t)V(t)$. The dominated convergence theorem now implies that as $\de(x)\to 0$
        \begin{align*}
        	\frac{V(\de(x))\int_{\de(x)}^{\diam D}\frac{U(t)}{t V(t)\psi(V(t)^{-2})}\, dt\,}{\PDS\sigma(x)}\lesssim \int_{V(\de(x))}^{\diam D}U(t)V(\de(x))dt\to0,
        \end{align*}
        which proves the claim.
    \end{proof}

\section{Linear Dirichlet problem for $\Lo$}\label{s:linear problem}
	 We begin this section with the definition of a weak-dual solution, a notion first introduced as a very weak solution by Brezis, and later developed and used in many state-of-the-art papers, see e.g. \cite{Aba15a,Aba17,AGCV19,BFV-cvPDE18,CGcV-jfa21,BVW-na21,Bio23cpaa}. The main advantage of this approach is that it eliminates the need to define the operator on a class of test functions. We will later see that essentially the weak-dual solution to \eqref{linear problem} is also a distributional solution.
	
	\begin{defn}\label{d:problem definition}
		Let $\lambda$ be a (signed) measure on $D$ and $\zeta$ a (signed) measure on $\partial D$ such that  
		\begin{align}\label{eq:condition on rate of solution}
			\int_DV(\de(x))|\lambda|(dx)+|\zeta|(\partial D)<\infty.\tag{$I_V$}
		\end{align}
        A function $u\in L^1_{loc}(D)$ is a weak-dual solution to the problem
		\begin{align}\label{linear problem}
			\begin{cases}
				\Lo u=\lambda,&\textrm{in $D$},\\
				\frac{u}{\PDS\sigma}=\zeta,&\textrm{on $\partial D$},
			\end{cases}
		\end{align}
		if for every $\xi\in C_c^\infty(D)$ it holds that
		\begin{align}\label{eq:problem - integral definition}
			\int_D u(x)\xi(x)dx=\int_D\GDP\xi(x)\lambda(dx)-\int_{\partial D}\frac{\partial}{\partial V}\GDP\xi(z)\zeta(dz).
		\end{align}
		If we have $\le$ in \eqref{eq:problem - integral definition} instead of the equality and the inequality holds for every non-negative $\xi\in C_c^\infty(D)$, then we say $u$ is a weak-dual subsolution \eqref{linear problem}. If the same holds for $\ge$ instead of $\le$, we call the function $u$ a weak-dual supersolution.	\end{defn}
	
	\begin{rem}
		\begin{enumerate}[label=(\alph*)]
			\item
			Let $\xi\in C_c^\infty(D)$. Note that, by calculations as in Proposition \ref{p:Poisson kernel}, the function $\frac{\partial}{\partial V}\GDP\xi(z)$ is well defined and
            $$-\frac{\partial}{\partial V}\GDP\xi(z)=\int_D\PDS(y,z)\xi(y)dy, \quad z \in \partial D.$$
			Further, \eqref{eq:Poisson kernel sharp} implies $\frac{\partial}{\partial V}\GDP\xi\in L^\infty(\partial D)$, and by Lemma \ref{l:Green invariant} we have $|\GDP\xi(x)|\lesssim V(\de(x))$, so the condition \eqref{eq:condition on rate of solution} ensures that the integrals in \eqref{eq:problem - integral definition} are finite.
			\item
			If $u$ is a solution to the linear problem \eqref{linear problem}, then $u=\GDP\lambda+\PDS\zeta$, a.e.~in $D$. Furthermore, $u\in \calL$ since both terms are in $\calL$. Conversely, by using Fubini's theorem, it is easily shown that the function $u=\GDP\lambda+\PDS\zeta$ 
            is a solution to \eqref{linear problem}.
            \item The relation \eqref{eq:problem - integral definition} can be also rewritten as follows:
            \begin{align*}
			\int_D u(x)\Lo\varphi(x)dx=\int_D\varphi(x)\lambda(dx)-\int_{\partial D}\frac{\partial}{\partial V}\varphi(z)\zeta(dz),
		\end{align*}
        for all $\varphi\in\{\eta\in C^V(D): \exists\xi \in C_c^\infty(D)\textrm{ such that }\Lo \eta=\xi \textrm{ in $D$, } \eta_{|\partial D}\equiv 0\}.$
		\end{enumerate}
	\end{rem}

 	This remark, together with Lemma \ref{l:Green invariant} and Lemma \ref{l:integrability of Poisson potentials} yields the following corollary.
	\begin{cor}\label{t:linear problem}
		Let $\lambda\in \MM(D)$ and $\zeta\in \MM(\partial D)$ such that  
		\eqref{eq:condition on rate of solution} holds. Then the linear problem \eqref{linear problem} has a unique weak-dual solution $u$. Furthermore, $u\in\LLL$,
		\begin{align*}
			u(x)=\GDP\lambda(x)+\PDS\zeta(x),\quad\text{for a.e. $x\in D$,}
		\end{align*}
		and there is $C=C(d,D,\psi,\phi)>0$ such that
		\begin{align*}
			\|u\|_{\LLL}\le C\left(\int_DV(\de(x))|\lambda|(dx)+|\zeta|(\partial D)\right).
		\end{align*}
	\end{cor}
	
An elementary observation gives us the following weak maximum principle.
	\begin{cor}[Weak maximum principle]\label{c:maximum principle weak solution}
	If $\lambda\ge0$ and $\zeta\ge0$ and they satisfy \eqref{eq:condition on rate of solution}, then the unique solution $u$ to the linear problem $\eqref{linear problem}$ satisfies $u\ge0$ a.e. in $D$.
	\end{cor}

	\begin{defn}\label{r:defn of distributional solution}
        A function $u\in \calL$ is a distributional solution to \eqref{linear problem} if for every $\xi\in C^\infty_c(D)$ it holds that
		\begin{align*}
			\int_D u(x)\Lo\xi(x)dx=\int_D \xi(x)\lambda(dx),
		\end{align*}
		and if for every $\varphi \in C(\overline D)$ it holds that
		\begin{align}\label{eq:distri solution boundary}
			\lim_{t\downarrow 0}\frac1t \int_{\{\de(x)\le t\}}\frac{u(x)}{\PDS\sigma(x)}\varphi(x)dx=\int_{\partial D}\varphi(z)\zeta(dz).
		\end{align}
	\end{defn}
	
	\begin{prop}\label{p:distSol}
		Let $\lambda\in \MM(D)$ and $\zeta\in L^1(\partial D)$ such that \eqref{eq:condition on rate of solution} holds. Then the weak-dual solution to \eqref{linear problem} is also a distributional solution to \eqref{linear problem}.
	\end{prop}
	\begin{proof}
		The weak-dual solution has the representation $u=\GDP \lambda + \PDS\zeta$, so the claim follows by Theorem \ref{t:GDPf dist-sol}, Propositions  \ref{p:harmonic connect}, \ref{p:bnd-bhv-PDSzeta} and \ref{p:boundary operator GD lambda}.\footnote{We need Prop \ref{p:bnd-bhv-PDSzeta} and \ref{p:boundary operator GD lambda} for \eqref{eq:distri solution boundary}.}
	\end{proof}
	
When the boundary measure has a density, $\zeta(dz)=\zeta(z)dz$, one can consider the pointwise version of the weak-$L^1$ boundary condition \eqref{eq:distri solution boundary} in the form of $\lim_{x\to z}\frac{u(x)}{\PDS\sigma}=\zeta(z)$. Obviously, the pointwise boundary condition is stronger than \eqref{eq:distri solution boundary} and is satisfied only in specific cases. When $\zeta\in C(\partial D)$ and $\lambda(dx)=f(x)dx$ with $f$ satisfying the conditions \ref{U1}-\ref{U4}, we can obtain the pointwise limit from \eqref{eq:1134b} and \eqref{e:1040}. 
 
 \subsection{Dirichlet forms and Kato's inequality}
 One of the key ingredients in solving semilinear elliptic equations is Kato's inequality, see \cite{AbaDup-Nonhomog2017,CaffarelliKato,chen_veron,DGcV-Kato,Figalli-DeGiorgi} for various fractional settings. Here, the standard ingredient in the proof of Kato's inequality is the use of regularity results for fractional spaces. In this section, we prove Kato's inequality in the non-local setting of our paper, using a general approach via Dirichlet forms. A similar method has been used in \cite{HuynhNguyen2022}, with the major difference being that we do not require an explicit form of the Dirichlet domain corresponding to the operator. As a reference for the results regarding general Dirichlet forms theory in this subsection, we refer to \cite{FOT}.
 
Recall that the operator $\Lo$, as an $L^2(D)$ operator, is given by
\begin{align*}
    \Lo u(x)=\sum_{j=1}^\infty \phi(\lambda_j)\wh u_j\varphi_j(x),
\end{align*}
for all $u\in \DD(\Lo)=\{u\in L^2(D): \sum_{j=1}^\infty \phi(\lambda_j)^2|\wh u_j|^2<\infty\}$. Moreover, Proposition \ref{p:Lo pointwise} gives the integral representation
\begin{align*}
    \Lo u(x)=\textrm{P.V.}\int_D(u(x)-u(y))J_D(x,y)dy+\kappa(x)u(x), \quad u\in C_c^\infty(D),
\end{align*}
and for all $u,v\in C_c^\infty(D)$ we have
\begin{align*}
\begin{split}
    \la\Lo u,v\ra&=\frac12\iint_{D\times D}(u(x)-u(y))(v(x)-v(y))J_D(x,y)dxdy+\int_D\kappa(x)u(x)v(x)dx.
\end{split}
\end{align*}
It is known that the operator $\Lo$ generates the Dirichlet form
\begin{align*}
    \EE(u,v)&=\int_0^\infty\int_D u(x)(v(x) -P_s^D v(x))dx\, \nu(s)ds\\
    &=\frac12\iint_{D\times D}\big(u(x)-u(y)\big)\big(v(x)-v(y)\big)J_D(x,y)dxdy+\int_D\kappa(x)u(x)v(x)dx
\end{align*}
with domain $\DD(\EE)= \{u\in L^2(D): \EE(u,u)<\infty\}$, which is regular with a core $C_c^\infty(D)$, see for example \cite[Section 13.4]{bernstein} and \cite[Theorem 2.1]{Okura}. Note that $\DD(\EE)$ is a normed space with the norm
\begin{align}\label{eq:EDB norm}
    \| u \|_{\DD(\EE)}=\| u \|_{L^2(D)}+\sqrt{\EE(u,u)}.
\end{align}
Since $\Lo$ has positive eigenvalues, for all $u\in C_c^\infty(D)$ it holds that 
\begin{align}\label{eq:coercive}
    \EE(u,u)=\la \Lo u, u\ra =\sum_{j=1}^\infty \phi(\lambda_j)|\wh u_j|^2\ge \phi(\lambda_1)\| u\|^2_{L^2(D)}.
\end{align}
By density of $C_c^\infty(D)$ in $(\DD(\EE),||\cdot||_{\DD(\EE)})$, condition \eqref{eq:coercive} holds on $\DD(\EE)$, so the Dirichlet form $(\EE,\DD(\EE))$ is coercive. This implies that $u\mapsto \sqrt{\EE(u,u)}$ is a norm on $\DD(\EE)$ which is equivalent to the norm \eqref{eq:EDB norm}.

\begin{defn}
    We say that $u\in L^2(D)$ is a variational solution to $\Lo u=f$ if $u\in \DD(\EE)$ and if $\EE(u,v)=\la f,v \ra$, for all $v\in \DD(\EE)$.
\end{defn}
A variational solution  $u$ to  $\Lo u=f$ also satisfies the inequality
\begin{align*}
	\EE(u,u)\lesssim \|f\|^2_{L^2(D)},
\end{align*}
which is a direct consequence of the Cauchy-Schwartz inequality and coercivity of $(\EE,\DD(\EE))$.  The following representation of the variational solution via Green potentials is essential for establishing the connection with the weak-dual solution to $\Lo u=f$.
\begin{prop}\label{p:varSol}
    For every $f\in L^2(D)$, there exists a unique variational solution $u$ to the equation $\Lo u=f$, and it holds that $u=\GDP f$.
\end{prop}
\begin{proof}
     This is a consequence of \cite[Theorem 1.5.4.]{FOT}, see also \cite[Theorem 1.5.2.]{FOT}, and Remark \ref{r:infty-gener}.
    
   
\end{proof}
This result, together with Corollary \ref{t:linear problem} and Proposition \ref{p:distSol}, gives the connection to the homogeneous linear problem (in the weak-dual and distributional sense), which we use to establish Kato's inequality.

\begin{cor}

    For $f\in L^2(D)$, the variational solution to $\Lo u=f$ in $D$ is the weak-dual and distributional solution to 
    \begin{align}\label{eq: homogeneous}
         \begin{cases}
         \Lo u=f,& \text{in $D$,}\\
         \frac{u}{\PDS\sigma}=0,&\text{on $\partial D$}.
    \end{cases}
     \end{align}
\end{cor}

 \begin{lem}\label{l:Kato}
     Let $\Lambda\in C^2(\R)$ be a convex function such that $\Lambda(0)=\Lambda'(0)=0$. Let $u$ be a weak-dual solution to \eqref{eq: homogeneous}
     for $f\in \LLL$. Then it holds
     $$\Lambda(u)\le \GDP\left(\Lambda'(u)f\right).$$
 \end{lem}
 \begin{proof}
     Assume first that $f\in C_c^\infty(D)$. Thus, $u=\GDP f\in L^\infty(D)$ is both the variational and the weak-dual solution to $\Lo u=f$, and it holds that
     \begin{align}\label{eq:varSol1340}
         \EE(u,v)=\la f,v\ra,\quad v\in \DD(\EE).
     \end{align}
     Further, $\Lambda\in C^2(\R)$ is convex, so $\Lambda'$ is Lipschitz, and there exists $C=C(\Lambda,f)>0$ such that for all $x,y\in D$ we have
     \begin{align*}
     	\begin{split}
     	    |\Lambda(u(x))-\Lambda(u(y))|&\le C |u(x)-u(y)|,\\
        |\Lambda'(u(x))-\Lambda'(u(y))|&\le C |u(x)-u(y)|,
     	\end{split}\quad
        \begin{split}
            |\Lambda(u(x))|&\le C |u(x)|,\\
         |\Lambda'(u(x))|&\le C |u(x)|.
        \end{split}
     \end{align*}     
     In other words, both $\Lambda/C$ and $\Lambda'/C$ are normal contractions, by the definition as in \cite[Eq. ($\EE$.4)'' on p.~5]{FOT}, so by \cite[Theorem 1.4.1(e)]{FOT} it holds that  $\Lambda(u),\Lambda'(u(x))\in \DD(\EE)$, and trivially $\Lambda(u),\Lambda'(u(x))\in L^\infty(D)$.
     For $v=\Lambda'(u)\GDP \xi$, where $\xi \in C_c^\infty(D)$, $\xi\ge0$, by  \cite[Theorem 1.4.2(ii)]{FOT} it follows that $v\in\DD(\EE)$. Now, by applying \eqref{eq:varSol1340} it follows that
     \begin{align}\label{eq:KatoLem1355}
         \int_D f(x) \Lambda'(u(x)) \GDP \xi(x) dx\ge \EE(\Lambda(u),\GDP\xi)=\int_D \Lambda(u(x))\xi(x)dx,
     \end{align}
     where the inequality is obtained by repeating the calculations as in \cite[Eqs. (5.4) - (5.6)]{HuynhNguyen2022}, and the equality comes from Proposition \ref{p:varSol} since $\Lambda(u)\in \DD(\EE)$. Moreover, $f \Lambda'(u)\in L_c^\infty(D)$ so we can use Fubini's theorem in the integral on the left-hand side in \eqref{eq:KatoLem1355}, to obtain
     \begin{align*}
         \int_D \GDP\big(f\, \Lambda'(u)\big)(x)  \xi(x) dx\ge \int_D \Lambda(u(x))\xi(x)dx,
     \end{align*}
     i.e. $\Lambda(u)\le \GDP\big(f\, \Lambda'(u)\big)$ a.e. in $D$.
     
     For general $f\in \LLL$, take an approximation sequence $(f_n)_n\subset C_c^\infty(D)$ such that $f_n\to f$ in $\LLL$ and almost everywhere. Define $u_n=\GDP f_n$, $n\in \N$ and note $u_n\to u$ in $\LLL$ and a.e. Then apply the dominated convergence theorem to get
     \begin{align*}
     	\Lambda(u)=\lim_n\Lambda(u_n)\le \lim_n \GDP\big(f_n\, \Lambda'(u_n)\big)= \GDP\big(f\, \Lambda'(u)\big).
     \end{align*}
 \end{proof}
 
 \begin{prop}[Kato's inequality]
     Let $u$ be a weak-dual solution to 
     \begin{align*}
         \begin{cases}
         \Lo u=f,& \text{in $D$,}\\
         \frac{u}{\PDS\sigma}=0,&\text{on $\partial D$}.
    \end{cases}
     \end{align*}
     Then it holds
     $$u^+\le \GDP\left(\1_{\{u> 0\}}f\right).$$
 \end{prop}
 \begin{proof}
     The proof is a standard application of Lemma \ref{l:Kato} to approximating sequence of convex functions $\Lambda_n(x) \to (x)^+$, as 
$n\to \infty$, see e.g. \cite[Proposition 5.4]{Bio23cpaa}.
\end{proof}

\section{The semilinear problem}\label{s:semilinear}
In this section, we study semilinear problems for $\Lo$. Let $f:D\times \R\to \R$ be the nonlinearity, and let $\zeta$ be a finite signed measure on $\partial D$. We  find a solution to
\begin{align}\label{eq:semilinear}
	\begin{cases}
		\Lo u(x)=f(x,u(x)),& x\in D,\\
		\frac{u}{\PDS\sigma}=\zeta,&\text{on $\partial D$},
	\end{cases}
\end{align}
in the weak-dual sense, i.e. a function $u\in L^1_{loc}(D)$ such that for all $\xi\in C_c^\infty(D)$,
\begin{align*}
	\int_D u(x)\xi(x)dx=\int_D\GDP\xi(x)f(x,u(x))dx-\int_{\partial D}\frac{\partial}{\partial V}\GDP\xi(z)\zeta(dz).
\end{align*}
The results of this section follow a similar line as in \cite{Bio23cpaa} for the subordinate spectral Laplacian. The proofs are therefore compactly presented, where only the major differences are highlighted. These differences originate in the preparatory results established in the preceding sections, which are inextricably intertwined with the arguments developed here.

At the beginning of the section, we bring forward two consequences of Kato's inequality. The proofs are standard and left to the reader; for further details we refer to e.g. \cite[Proposition 5.8 \& Corollary 5.7]{Bio23cpaa}.
\begin{cor}\label{c:uniq-noninc}
	Let $\zeta$ be a finite signed measure on $\partial D$. If the nonlinearity $f:D\times \R\to \R$ is non-decreasing in the second variable, then the semilinear problem \eqref{eq:semilinear} has at most one weak-dual solution.
\end{cor}

\begin{cor}
	Let $u_1$ and $u_2$ be weak-dual solutions to the semilinear problem \eqref{eq:semilinear}.
	Then $u=\max\{u_1,u_2\}$ is a subsolution to \eqref{eq:semilinear}.
\end{cor}

 \subsection{Existence results for semilinear problems}
  Throughout this subsection, we impose the following behaviour of the nonlinearity:
 	\begin{assumption}{F}{}\label{F}
		The function  $f:D\times \R\to\R$ is continuous in the second variable, and there exists a locally bounded function $q:D\to[0,\infty]$ and  a   non-decreasing function $\Lambda:[0,\infty)\to[0,\infty)$ such that $|f(x,t)|\le q(x)\Lambda(|t|)$, $x\in D$, $t\in \R$.
	\end{assumption}
Also, we will often use the shortened notation of Nemytskii operators $f_u(x)\coloneqq f(x,u(x))$. As a cornerstone result, first we prove the method of sub- and supersolutions for homogeneous semilinear problems. 
\begin{thm}\label{t:super-subsol} 
		Let $f$ satisfy $\ref{F}$. Assume that there exists a  supersolution $\overline u$ and a subsolution $\underline{u}$ to the semilinear problem
		\begin{align}\label{eq:semilinear sub super problem}
			\begin{cases}
				\Lo u(x)=f(x,u(x)),&\textrm{in $D$},\\
				\frac{u}{\PDS\sigma}=0,&\textrm{on $\partial D$},
			\end{cases}
		\end{align}
		of the form $\underline u=\GDP \underline h$ and $\overline u=\GDP \overline h$ such that $\underline u\le \overline u$, $\underline h(x)\le f(x,\underline u(x))$ and $f(x,\overline u(x))\le \overline h(x)$ a.e. in $D$, and $\overline{u},\underline{u}\in\LLL \cap L^\infty_{loc}(D)$. Further, assume that $q\Lambda(|\underline u|\vee |\overline u|)\in \LLL$. Then there exist weak-dual solutions $u_1,u_2\in\LLL$ to \eqref{eq:semilinear sub super problem} such that every solution to \eqref{eq:semilinear sub super problem} with property $\underline{u}\le u \le \overline u$  satisfies
		$$ \underline u\le u_1 \le u\le u_2\le \overline u.$$
		
		Further, every weak-dual solution $u$ of \eqref{eq:semilinear sub super problem} with property $\underline{u}\le u \le \overline u$ is continuous after the modification on a Lebesgue null set. Additionally, if the nonlinearity $f$ is non-increasing in the second variable, the weak-dual solution to \eqref{eq:semilinear sub super problem} is unique.
	\end{thm}
 \begin{proof}
    We briefly present the proof, since it is standard and essentially the same as the proof of \cite[Theorem 5.9]{Bio23cpaa}, where $\Lbase=\Delta_{|D}$.
    
     \textit{Step 1: existence of a solution.} Define $F:D\times \R\to \R$ by 
     \begin{align*}
     	F(x,t)=\begin{cases}
     		f(x,\underline u(x)),&t<\underline u(x),\\
     		f(x,t),&\underline{u}\le t \le \overline u,\\
     		f(x,\overline u(x)),&\overline u(x)<t,
     	\end{cases}
     \end{align*}
     and the operator $\KK :\LLL\to \LLL$ by
     \begin{align*}
     	\KK v(x)=\int_D\GDP(x,y)F(y,v(y))dy,\quad x\in D,\, v\in \LLL.
     \end{align*}
     By using the Arzel\`{a}-Ascoli theorem, it follows that $\KK$ is a compact operator so, by Schauder's fixed point theorem, there exists a function $u\in \LLL$ such that $\KK u=u$. Hence $u$ is a weak-dual solution to
     \begin{align*}
     	\begin{cases}
     		\Lo u(x)=F(x,u(x)),&\textrm{in $D$},\\
     		\frac{u}{\PDS\sigma}=0,&\textrm{on $\partial D$}.
     	\end{cases}
     \end{align*}
     By applying Kato's inequality, it follows that $\underline u \le u \le \overline u$, so $F(x,u(x))=f(x,u(x))$, hence $u$ solves \eqref{eq:semilinear sub super problem}.
     
     \textit{Step 2: finding the minimal and maximal solution.} We use Zorn's lemma on the family  $\mathcal{P}\coloneqq\{u\in \LLL: \underline u\le u\le \overline u\text{ and $u$ solves \eqref{eq:semilinear sub super problem}}\}$. Here we need the  uniformness result from Proposition \ref{p:uniform} to show that there exists a supremum of every totally ordered subset of $\mathcal{P}$.
     
     \textit{Step 3: continuity of solutions.} Since every weak-dual solution is of the form $u=\GDP f_u$, the claim follows from the dominated convergence theorem by the assumption $q\Lambda(|\underline u|\vee |\overline u|)\in \LLL$ and Proposition \ref{p:uniform}.
     
     \textit{Step 4: uniqueness of the solution.} When $f$ is non-increasing in the second variable, uniqueness follows from Corollary \ref{c:uniq-noninc}.
\end{proof}

By using the method of sub- and supersolution, we are able to prove the existence of a solution to the semilinear problem in the case of non-positive nonlinearities. Another important ingredient in the proof is the approximation of non-negative harmonic functions with a sequence of Green potentials, the so-called Hunt approximation, which allows us to deal with nonhomogeneous boundary conditions.

 \begin{thm}\label{t:non-positive semilinear problem}
		Let $f:D\times \R\to(-\infty,0]$ satisfy \ref{F} and $f(x,0)=0$ for all $x\in D$. Further, let $\zeta\in \MM(\partial D)$ be a finite non-negative measure such that
		\begin{equation}\label{eq:semilinear integral condition}
		    q\Lambda(\PDS\zeta)\in\LLL.
		\end{equation}
		Then the problem \eqref{eq:semilinear}, i.e.
		\begin{align*}
			\begin{cases}
				\Lo u(x)=f(x,u(x)),&\textrm{in $D$},\\
				\frac{u}{\PDS\sigma}=\zeta,&\textrm{on $\partial D$},
			\end{cases}
		\end{align*}
		has a weak-dual solution $u\in C(D) \cap \LLL$. Additionally, if $f$ is non-increasing in the second variable, the continuous weak-dual solution is unique.
	\end{thm}
 \begin{proof}
     The proof   follows the same steps as in the proof of  \cite[Theorem 5.10]{Bio23cpaa}. Here we present just a sketch, while the rigorous computations can be easily adapted from \cite[Theorem 5.10]{Bio23cpaa}, by applying the preparatory results from previous sections.
     
     First we approximate the nonhomogeneous problem \eqref{eq:semilinear} with appropriately chosen homogeneous problems. Take $(\wt f_k)_k$, a sequence of non-negative  bounded functions such that $\GDP \wt f_k \uparrow \PDS \zeta$ in $D$. The existence of such a sequence is given by applying the procedure in \cite[Appendix A.1]{BVW-na21}, where we note that the semigroup $(R_t^D)_t$, induced by $Y$, is transient and strongly Feller, that $\GDP V(\de)\asymp V(\de)$ by Lemma \ref{l:Green invariant}, and that $\PDS\zeta$ is a continuous and harmonic with respect to $Y$ by Proposition \ref{p:harmonic connect}.
     
     The auxiliary homogeneous problem
     \begin{align*}
     	\begin{cases}
     		\Lo u(x)=f(x,u(x))+\wt f_k,&\textrm{in $D$},\\
     		\frac{u}{\PDS\sigma}=0,&\textrm{on $\partial D$}.
     	\end{cases}
     \end{align*}
     admits a solution $u_k=\GDP f_{u_k}+\GDP \wt f_k$ by Theorem \ref{t:super-subsol}.
     By the Arzel\`{a}-Ascoli theorem, we find a subsequence of $(u_ k)_k$ which converges to a solution $u$ to \eqref{eq:semilinear}. 
     
 \end{proof}

	If the nonlinearity is non-negative, we can also consider the nonhomogeneous problem \eqref{eq:semilinear} and obtain a solution, under a slightly stronger condition than \eqref{eq:semilinear integral condition}.

 \begin{thm}\label{t:semilin non-negative monotone linearity}
		Let $f:D\times \R\to[0,\infty)$ satisfy \ref{F}, and let $f$ be a non-decreasing function in the second variable. Let $\zeta$ be a non-negative finite measure on $\partial D$ such that
		\begin{align*}
			\GDP\big(q\Lambda(2\PDS\zeta)\big)\le \PDS\zeta,\quad \text{in $D$.}
		\end{align*}
		 Then  there is a continuous non-negative solution to
		\begin{align*}
			\begin{cases}
				\Lo u(x)=f(x,u(x)),&\textrm{in $D$},\\
				\frac{u}{\PDS\sigma}=\zeta,&\textrm{on $\partial D$}.
			\end{cases}
		\end{align*}
	\end{thm}
	
	\begin{proof}
		The claim easily follows  by the method of monotone iterations, cf. \cite[Theorem 5.14]{Bio23cpaa}.
	\end{proof}

	For a signed nonlinearity, we also have an existence result in a slightly weaker form.
 	\begin{thm}\label{t:semilinear signed data}
		Let $f:D\times \R\to \R$ satisfy \ref{F} and let $\zeta$ be a finite measure on $\partial D$. Assume that $\GDP q\in C_0(D)$ and $\GDP\big(q\Lambda(2\PDS|\zeta|)\big)\in C_0(D)$. Assume additionally that: (a) $\Lambda$ is sublinearly increasing, i.e. $\lim_{t\to\infty}\Lambda(t)/t=0$, or  (b) $m>0$ is sufficiently small. Then the semilinear problem
		\begin{align}\label{eq:semilinear bogdan}
			\begin{cases}
				\Lo u(x)=m\,f(x,u(x)),&\textrm{in $D$},\\
				\frac{u}{\PDS\sigma}=\zeta,&\textrm{on $\partial D$}.
			\end{cases}
		\end{align}
		has a weak continuous solution $u$ such that $|u|\le C+\PDS|\zeta|$, for some constant $C\ge0$.
		
		If, in addition, $f$ is non-increasing in the second variable, $u$ is a unique weak-dual solution to \eqref{eq:semilinear bogdan}.
	\end{thm}
	\begin{proof}
		The proof follows the calculations as in the proof of \cite[Theorem 5.16]{Bio23cpaa} almost to the letter. The details are left to the reader.
	\end{proof}

    The results that we obtained in the previous theorems also allow us to consider a simple nonlinearity $f(x,t)=\lambda t$, for $\lambda\in \R$. However, at this point we can only extrapolate the existence of a solution for $\lambda\in(-\infty,m)$, where $m>0$ is small enough (for negative $\lambda$ we use Theorem \ref{t:non-positive semilinear problem} and for positive $\lambda$ we use Theorem \ref{t:semilin non-negative monotone linearity}). However, the existence result is known in the interpolated fractional setting for all non-spectral $\lambda$, i.e. for all $\lambda \in \R\setminus\{\psi(\lambda_j):j\in \N\}$, with nonhomogeneous boundary condition. Such problems give rise to the so-called large eigenvalues. Moreover, for each eigenvalue, there exist infinitely many essentially different large eigenfunctions, reflecting a purely nonhomogeneous, non-local phenomenon. For details, we refer to \cite{CGcV-jfa21}. We believe that this result can also be obtained in our setting, but we need some additional properties of improved integrability as those in \cite[Section 3]{CGcV-jfa21}, partly obtained in Appendix \ref{ap:GR-reg}.  We leave such calculations for future work.

	\subsubsection*{Semilinear problems for the interpolated fractional Laplacian}
	In this subsection, we assume that $\Lo=\INTFR$, i.e. we are in the case when $\psi(\lambda)=\lambda^{\alpha/2}$ and $\phi(\lambda)=\lambda^{\beta/2}$, for $\alpha,\beta\in (0,2)$. We will apply our general results from the previous subsection to this special case, for nonlinearities $f$ that are power-like, i.e. where the functions $q$ and $\Lambda$ in \ref{F} are of the form $q(x)=\de(x)^\theta$ and $\Lambda(t)=|t|^p$. The existence and nonexistence results are derived in terms of powers $\theta$, $p$, $\alpha$ and $\beta$.
	
	The first theorem gives a complete existence characterization, when the problem is determined by  a non-positive nonlinearity and non-negative boundary datum $\zeta$, in terms of the critical value for $p$.
	\begin{thm}\label{t:frac-semi-neg}
		Let $f(x,t)=-\de(x)^\theta|t|^p$, for some $\theta\in \R$ and $p>0$, and assume that $\Lo=\INTFR$. If $p<\frac{1+\frac{2\theta}{2+\beta}}{1-\frac{\beta\alpha}{2+\beta}}$, then the problem
		\begin{align*}
			\begin{cases}
				\Lo u(x)=f(x,u(x)),&\textrm{in $D$},\\
				\frac{u}{\PDS\sigma}=\zeta,&\textrm{on $\partial D$},
			\end{cases}
		\end{align*}
		has a non-negative continuous weak-dual solution for every non-negative function $\zeta\in C(D)$. Moreover, the boundary condition holds pointwisely.
		
		If, on the other hand, $p\ge\frac{1+\frac{2\theta}{2+\beta}}{1-\frac{\beta\alpha}{2+\beta}}$ and if $\zeta\in C(\partial D)$, $\zeta\not\equiv 0$, is non-negative, then the weak-dual solution to \eqref{eq:semilinear}, such that the boundary condition holds pointwisely, does not exists.
	\end{thm}
	\begin{proof}
		By the sharp bounds for $\PDS\sigma$ in this specific case, provided in  Corollary \ref{c:PDS-inter-sharp}, it follows that we can apply Theorem \ref{t:non-positive semilinear problem} if and only if $p<\frac{1+\frac{2\theta}{2+\beta}}{1-\frac{\beta\alpha}{2+\beta}}$. This provides us with a non-negative solution such that $u=\GDP f_u+\PDS\zeta$ and $0\le u\le \PDS\zeta$. In particular, $0\le -\GDP f_u\lesssim -\GDP f_{\PDS\sigma}$ so the pointwise boundary condition $u/\PDS\sigma\to \zeta$ follows by Propositions \ref{p:bnd-bhv-PDSzeta} and \ref{p:boundary operator GD lambda}.
		
		Let us now look at the case when $p\ge\frac{1+\frac{2\theta}{2+\beta}}{1-\frac{\beta\alpha}{2+\beta}}$. Assume that $\zeta\in C(\partial D)$, $\zeta\not\equiv 0$, is non-negative, and that there exists a solution $u$ to \eqref{eq:semilinear} for which the boundary condition holds pointwisely. Then, by Proposition \ref{p:bnd-bhv-PDSzeta}, we have $u(x)\gtrsim \PDS\sigma\asymp \de(x)^{-1-\beta/2+\alpha\beta/2}$, at least near some small but fixed part of $\partial D$ where $\zeta>0$. Denote this set by $D_\Gamma$, so we have, by \eqref{eq:GDP sharp}, for all points $x\in D$ away from the boundary 
		\begin{align*}
			-\GDP(f_u)(x)&=-\int_D \GDP(x,y)f_u(y)dy\gtrsim \int_{D_\Gamma}\GDP(x,y) \de(y)^{\theta+p(-1-\beta/2+\alpha\beta/2)}dy\\
			 &\ge c(\de(x))\int_{D_\Gamma}\de(y)^{\beta/2} \de(y)^{\theta+p(-1-\beta/2+\alpha\beta/2)}dy\\
			 &\asymp c(\de(x))\int_0^1 t^{\theta+p(-1-\beta/2+\alpha\beta/2)+\beta/2}=\infty
		\end{align*}
		since $\theta+p(-1-\beta/2+\alpha\beta/2)+\beta/2\ge -1$,	 which contradicts the definition of the weak-dual solution.
	\end{proof}
	
		In the case of non-negative nonlinearities, we have the following theorem.
	\begin{thm}\label{t:frac-semi-pos}
		Let $f(x,t)=m\delta(x)^\theta|t|^p$, for some $m,p>0$, $\theta\in \R$, and assume $\Lo=\INTFR$. If $p<\frac{1+\frac{2\theta}{2+\beta}}{1-\frac{\beta\alpha}{2+\beta}}$ and $m$ is sufficiently small, then the problem
		\begin{align*}
			\begin{cases}
				\Lo u(x)=f(x,u(x)),&\textrm{in $D$},\\
				\frac{u}{\PDS\sigma}=\zeta,&\textrm{on $\partial D$},
			\end{cases}
		\end{align*}
		has a continuous non-negative solution for all non-negative $\zeta\in C(\partial D)$, and the solution satisfies the boundary condition pointwisely.
		
		On the other hand, if $p\ge\frac{1+\frac{2\theta}{2+\beta}}{1-\frac{\beta\alpha}{2+\beta}}$, then the problem \eqref{eq:semilinear} does not have a solution for any choice of $m>0$ and any non-negative  $\zeta\in C(\partial D)$, $\zeta\not\equiv 0$.
	\end{thm}
	\begin{proof}
		If $p<\frac{1+\frac{2\theta}{2+\beta}}{1-\frac{\beta\alpha}{2+\beta}}$ and $m$ is sufficiently small, then the assumptions of Theorem \ref{t:semilin non-negative monotone linearity} are satisfied since $\PDS \zeta \lesssim \PDS\sigma$, and due to Corollary \ref{c:bndry behav of G(de^k)}, so we obtain a non-negative solution. Since $f_u\lesssim f_{\PDS\sigma}$, the boundary condition holds pointwisely by Propositions \ref{p:bnd-bhv-PDSzeta} and \ref{p:boundary operator GD lambda}.
		
		On the other hand, assume that $p\ge\frac{1+\frac{2\theta}{2+\beta}}{1-\frac{\beta\alpha}{2+\beta}}$ and that $\zeta\in C(\partial D)$, $\zeta\not\equiv 0$, is non-negative. Assume that $u$ is a solution to \ref{eq:semilinear}. Since $f(x,t)\ge 0$, $u\ge \PDS\zeta$, so  $u(x)\gtrsim \de(x)^{-1-\beta/2+\alpha\beta/2}$ near a part of $\partial D$ where $\zeta>0$, so we can repeat the reasoning from the end of the proof of Theorem \ref{t:frac-semi-neg}, to see that $\GDP f_u \gtrsim \GDP \de^{\theta+p(-1-\beta/2+\alpha\beta/2)}=\infty$, which contradicts the definition of the weak-dual solution.
	\end{proof}

	For a signed nonlinearity we give a partial answer, with a narrower range for the power parameters than in the previous two results.
	\begin{thm}\label{t:frac-semi-sign}
		Let $f(x,t):D\times \R\to \R$ such that $|f(x,t)|\le m\delta(x)^\theta|t|^p$, for some $m,p>0$ and $\theta> -\frac{\alpha\beta}{2}$, and assume  $\Lo=\INTFR$. If 
		\begin{itemize}
			\item $1\le p<\frac{\frac{2\theta+\alpha\beta}{2+\beta}}{1-\frac{\beta\alpha}{2+\beta}}$ and $m$ is sufficiently small, or
			\item $p<1$,
		\end{itemize}
		then the problem
		\begin{align*}
			\begin{cases}
				\Lo u(x)=mf(x,u(x)),&\textrm{in $D$},\\
				\frac{u}{\PDS\sigma}=\zeta,&\textrm{on $\partial D$},
			\end{cases}
		\end{align*}
		has a  continuous weak-dual solution for every  $\zeta\in C(\partial D)$, and the solution satisfies the boundary condition in the pointwise sense.
	\end{thm}
	\begin{proof}
		By Corollary \ref{c:bndry behav of G(de^k)}, $\GDP \de^\kappa \in C_0(D)$ if and only if $\kappa > -\frac{\alpha\beta}{2}$. Hence, $\GDP q\in C_0(D)$. Recall also that $\PDS|\zeta|\lesssim \PDS\sigma$ since $\zeta \in C(\partial D)$, so if $p<\frac{\frac{2\theta+\alpha\beta}{2+\beta}}{1-\frac{\beta\alpha}{2+\beta}}$, then $\GDP f(\PDS |\zeta|)\in C_0(D)$ and in that case we can apply Theorem \ref{t:semilinear signed data}. The obtained solution satisfies $|u|\le C+\PDS|\zeta|$, so the pointwise boundary condition follows from Propositions \ref{p:bnd-bhv-PDSzeta} and \ref{p:boundary operator GD lambda}.
	\end{proof}

\appendix
    \section{Appendix}
    \subsection{Joint continuity of the kernels}
	\begin{lem}\label{ap:l:joint p_D}
            The functions $p_D(t,x,y)$ is symmetric\footnote{$p_D(t,x,y)=p_D(t,y,x)$ for all $x,y\in \R^d$ and $t>0$.} and jointly continuous in $(0,\infty)\times \R^d\times  \R^d$.
            Furthermore, $p_D(t,x,y)/V(\de(y))$ is jointly continuous in $(0,\infty)\times \overline D\times  \overline D$, where the boundary values are extended by
            \begin{align*}
                \partial_V p_D(t,x,z)\coloneqq\lim_{D\ni y\to z}p_D(t,x,y)/V(\de(y)), \ z\in \partial D.
            \end{align*}
    \end{lem}
    \begin{proof}
    First we note that by assumption \ref{as:WSC} and the representation \eqref{eq:SBM dens}, it follows that $p(t,x,y)$ is jointly continuous in $(0,\infty)\times \R^d\times \R^d$. By Hunt's formula \eqref{eq:p_D Hunt}, the same holds for $p_D(t,x,y)$ in $(0,\infty)\times D\times D$. However, by the standard calculations as in \cite[Section 2]{chung_zhao}, we get that $p_D(t,x,y)=p_D(t,y,x)$ for all $(0,\infty)\times \R^d\times  \R^d$, and that $p_D$ continuously vanishes on $\partial D$ since $D$ is regular for the process $X$ since $D$ is $C^{1,1}$ regular. Hence, $p_D$ is jointly continuous in $(0,\infty)\times \overline D\times  \overline D$.
        
    By the spectral decomposition of $\Lbase$, we have $p_D(t,x,y)=\sum_{j=1}^{\infty} e^{-\lambda_j t}\varphi_j(x)\varphi_j(y),$
        for almost all $(t,x,y)\in (0,\infty)\times \overline D\times  \overline D$. The eigenfunctions $\varphi_j$, $j\in \N$, are also regular in the sense of Theorem \ref{t:eigenfuncions}, so this spectral representation holds everywhere. 
        Specifically, by \eqref{eq:eigen no2}, $\varphi_j /V(\de)$ are H\"older regular up to the boundary, so they can be continuously extended to the boundary. By this regularity, the corresponding bounds for the H\"older norms, and by the spectral representation, we get that $p_D(t,x,y)/V(\de(y))$ is also jointly continuous in $(0,\infty)\times \overline D\times  \overline D$, and in particular $\partial_V p_D(t,x,z)$ is jointly continuous in $(0,\infty)\times \overline D\times  \partial D$.
    \end{proof}
    \begin{rem}\label{r:Bio21jmaa}
        In this paper we rely on recent potential-theoretic developments related to the killed subordinate Brownian motion $X^D$ from \cite{Bio21jmaa}. Throughout that paper, it is assumed that the process $X$ is transient, see \cite[Subsection 2.4]{Bio21jmaa}. This is because the results were obtained also for unbounded domains, and to assure that the Green function of $X$ exists and is appropriately bounded: $G_{\R^d}(x,y)\le c(d,R)|x-y|^{-d}\phi(|x-y|^{-2})^{-1}$, for $|x-y|<R$. For bounded domains, this assumption was used only to bound $G_D(x,y)$ by $|x-y|^{-d}\phi(|x-y|^{-2})^{-1}$ from above, see \cite[p.9, and Eqs. (38) and (39)]{Bio21jmaa}. However, in the meantime, such a bound for $G_D$ was obtained without the transience assumption, see \cite[Lemma 2.1]{ksv2020boundary}. Thus, we may use \cite{Bio21jmaa} on bounded domains $D$ without assuming the transience of $X$.
    \end{rem}

    \begin{lem}\label{l:GDP cont}
        The Green functions  $G_D$ and $\GDP$ are  jointly continuous in $D\times D \setminus \{(x,x):x\in D\}$ and on the diagonal it holds that $\lim_{(z,w)\to(x,x)}\GD(z,w)=\lim_{(z,w)\to(x,x)}\GDP(z,w)=\infty$, $x\in D$,
        i.e. they are continuous in $D\times D$ in the extended sense.
    \end{lem}
    \begin{proof}
        First, we deal with $\GDP$. Recall that $p_D$ is jointly continuous in $(0,\infty)\times D\times D$ and that the sharp bounds \eqref{eq:heat-sharp-small} and \eqref{eq:heat-sharp-big} hold. Additionally, $\uu$ is integrable on $(0,1)$ and bounded on $[1,\infty)$, so we may use the dominated convergence theorem in the representation $\GDP(x,y)=\int_0^\infty p_D(t,x,y)\uu(t)dt$ to get joint continuity for $x,y\in D$, $x\ne y$. Thus, the claim follows from the sharp estimate \eqref{eq:GDP sharp}.

        Analogously, from the continuity of $p_D$ and the bounds \eqref{eq:heat-sharp-small} and \eqref{eq:heat-sharp-big}, we easily get continuity of $G_D$ in $x\ne y$. For $x=y$, sharp bounds for $G_D$ have not been proved, since $X$ is not necessarily transient. However, the claim follows from \cite[last display on p. 137]{ksv2020boundary}, since for $(x,y)$ close enough to $(x_0,x_0)$
        $$G_D(x,y)\ge G_{B(x_0,\de(x_0)/2)}(x,y)\ge c(d,\de(x_0),D,\phi)\frac{1}{|x-y|^dV(|x-y|)^{-2}}.$$ 
    \end{proof}
  
\subsection{Regularity of eigenfunctions of $\Lbase$}\label{ap:GR-reg}
\begin{proof}[Proof of Theorem \ref{t:eigenfuncions}]

		By general semigroup theory, the potential operator $G_D$ is the inverse operator of $\Lbase$ on $L^2(D)$, see e.g. \cite{sato_PO,yosida_FA}, so we have $G_D\varphi_j=\frac{1}{\lambda_j}\varphi_j$, $j\in\N$, a.e. in $D$.	Further, by the Green function estimate
		\eqref{eq:GD sharp bound SBM} and by \ref{as:WSC} we have $G_D(x,y)\lesssim
		|x-y|^{-d+2\ud{\phi}}$, $x,y\in D$. Hence, for $g$ supported in $D$ we have
		\begin{align*}
			|G_Dg(x)|\le \int_D G_D (x,y)|g(y)|dy\lesssim
			\int_{\R^d}\frac{|g(y)|}{|x-y|^{d-2\ud{\phi}}}dy\eqqcolon U_{\ud{\phi}}|g|(x),
		\end{align*}
		where $U_{\ud{\phi}}$ denotes the Green function of the $\ud{\phi}$-fractional Laplacian. Thus, we can repeat the procedure as in \cite[Proposition
		1.2]{Fernandez-RealRos-Oton2014}:  for $g\in L^2(D)$ we have
		\begin{itemize}
			\item if $1<p<\frac{n}{2\ud{\phi}}$, then
			 $\| G_D g \|_{L^q(D)}\le C(d,p,\ud{\phi}) \|g\|_{L^p(D)}$, where $q=\frac{np}{np-2p\ud{\phi}}$, 
			\item 
			if $p=\frac{n}{2\ud{\phi}}$, then
			$\| G_D g \|_{L^q(D)}\le C(d,p,\ud{\phi}) \|g\|_{L^p(D)}$, where $q<\infty,$
			\item 
			
			if $p>\frac{n}{2\ud{\phi}}$, then
			$ \| G_D g \|_{L^\infty(D)}\le C(d,p,\ud{\phi}) \|g\|_{L^p(D)}. $
		\end{itemize}
		Since $\varphi_j\in L^2(D)$, by bootstrapping the previous bounds, we obtain that $\varphi_j$ is bounded and the bound \eqref{eq:eigen no0} holds.
		
		Next, since $D$ is $C^{1,1}$ domain, all the points on $\partial D$ are regular for $X$, so by applying \cite[Lemma 2.2 \& Lemma 2.4]{BVW-na21} we get that $\varphi_j\in C_0(D)$. The
		inequalities \eqref{eq:eigen no1} and \eqref{eq:eigen no2} now follow by
		\cite[Theorem 1.1 \& Theorem 1.2]{KKLL-jfa19}.
\end{proof}
  
 \subsection{A note on the domain of killed operators}\label{ap:domain}
 
 	\begin{lem}\label{ap:l:generator domain}
		 The space $C_c^2(D)$ is contained in the $L^2(D)$-domain, but not in the $C_0(D)$-domain, of $\Lbase$. Moreover, for $f\in C_c^2(D)$ it holds that
		\begin{align*}
			\Lbase f=L f,\quad \textit{in $D$}.
		\end{align*}
	\end{lem}
	\begin{proof}
		Take $f\in C_c^2(D)$. First we show that 
		\begin{align}\label{eq:killed semigr conv}
			\lim_{t\to0}\frac{P^D_tf(x)-f(x)}{t}=L f(x),\quad x\in D,
		\end{align}
		which immediately implies that $L f$ is the only candidate for $-\Lbase f$ both in $L^2(D)$ and in $C_0(D)$. To this end, write
		\begin{align}\label{eq:semi decomp}
			\frac{P^D_tf(x)-f(x)}{t}=\frac{P_tf(x)-f(x)}{t}-\frac{\wt P^D_tf(x)}{t},
		\end{align}
		where
		\begin{align*}
			\wt P^D_tf(x)&\coloneqq \int_D \ex_x\left[p(t-T_D,X_{T_D},y);t>T_D\right]f(y)dy=\ex_x\left[P_{t-T_D}f(X_{T_D});t>T_D\right].
		\end{align*}
		The first term on the right-hand side in \eqref{eq:semi decomp} converges to $L f(x)$ uniformly in $\R^d$ since $C_c^2(\R^d)\subset\DD(L)$, and 
		$\left\|\frac{P_tf-f}{t}\right\|_{L^\infty(D)}\le C \|L f\|_{L^\infty(\R^d)}<\infty,$, $t<1$, see e.g. \cite[Eq. (13.3)]{bernstein}.
		
		For the second term, since $f\equiv 0$ on $D^c$ and by using the same arguments as above, we have that $\frac{P_{t-T_D}f(X_{T_D})\1_{\{t>T_D\}}}{t}$, for $t\in(0,1)$, is uniformly bounded $\p_x$-a.s.,  with the same bound for all $x\in D$, so $\frac{\wt P^D_tf}{t}$ is uniformly bounded in $D$, too. Notice that $\lim_{t\to0}\frac{P_{t-T_D}f(X_{T_D})\1_{\{t>T_D\}}}{t}\to 0$  $\p_x$-a.s. since the numerator becomes 0 for small $t$, so by the dominated convergence we have $\wt P^D_tf(x)/t\to0$ for all $x\in D$. This proves \eqref{eq:killed semigr conv}.
		
		To prove that $f$ is in the $L^2(D)$ domain of $\Lbase$, we use \eqref{eq:semi decomp} to get
		\begin{align*}
				\left\|\frac{P^D_tf-f}{t}-L f\right\|_{L^2(D)}\le 	C\left\|\frac{P_tf-f}{t}-L f\right\|^2_{L^\infty(D)}+\left\|\frac{\wt P^D_t f}{t}\right\|_{L^2(D)}.
		\end{align*}
		The first term goes to 0 because of the uniform convergence of $\frac{P_tf-f}{t}$ to $L f$, and the second goes to 0 by using the dominated convergence of $\wt P^D_tf(x)/t$ to 0 as we explained in the previous paragraph.
		
		We prove that  $f$ is not in the $C_0(D)$ domain of $\Lbase$, for $f\not\equiv0$, $f\ge0$ or $f\le0$. Recall from the beginning of the proof that $L f$ is the only candidate for $\Lbase f$. However, $L f\not\in C_0(D)$ since $L f(x)=\int_{D}f(y)j_\phi(|x-y|)dy$, for $x$ away from $\supp f$, and does not vanish at the boundary because $j_\phi>0$.
	\end{proof}

    \subsection{Auxiliary asymptotic behaviour}
    
    \begin{lem}\label{ap:l:rho}
        The weight $\rho(\de)$ satisfies the conditions \ref{U1}, \eqref{eq: U3'} and \ref{U4}.
        
    \end{lem}
    \begin{proof}
        The properties \ref{U1} and \ref{U4} are clear since $\rho$ is bounded. We derive \eqref{eq: U3'} for each of the terms in $\rho(\de)$. For the first term, $V(\de)^2\psi(V(\de)^{-2})$ we note that for $t\le s$ it holds that
    \begin{align*}
        \frac{V(t)^2\psi(V(t)^{-2})}{V(s)^2\psi(V(s)^{-2})}\le \oa{\psi}\frac{V(t)^2}{V(s)^2}\left(\frac{V(t)^{-2}}{V(s)^{-2}}\right)^{\od{\psi}}\le \oa{\psi}\ua{\phi}^{2-2\od{\psi}}\left(\frac{t}{s}\right)^{\ud{\phi}(2-2\od{\psi})},
    \end{align*}
    and similarly for the lower bound. The property \eqref{eq: U3'} follows similarly.

    For the second term take $c>0$, and $s,t\in (0,1/2]$ such that $c^{-1}t\le s\le c\,t$. If $s\le t$, we have $c_1^{-1}(t/s)^{\ud{\phi}}\le V(t)/V(s)\le c_1 (t/s)^{\od{\phi}}$, where $c_1=c_1(\phi)>0$ so
    \begin{align*}
    	\frac{V(t)\int_{V(t)}^1\psi(h^{-2})dh}{V(s)\int_{V(s)}^1\psi(h^{-2})dh}\le c_1 c\frac{\int_{V(t)}^1\psi(h^{-2})dh}{\int_{V(s)}^1\psi(h^{-2})dh}\le c_1 c,
    \end{align*} 
    since $V(s)\le V(t)$. For the lower bound, by using \eqref{eq:WSC} for $\psi$, we get
    \begin{align*}
    	\frac{V(t)\int_{V(t)}^1\psi(h^{-2})dh}{V(s)\int_{V(s)}^1\psi(h^{-2})dh} &\ge (c_1 c)^{-1}\frac{\int_{c_2V(s)}^1\psi(h^{-2})dh}{\int_{V(s)}^1\psi(h^{-2})dh}=(c_1 c)^{-1}\frac{\int_{V(s)}^{1/c_2}\psi((c_2h)^{-2})dh}{\int_{V(s)}^1\psi(h^{-2})dh}\\
    	&\ge c_3 \left(1-\frac{\int_{1/c_2}^{1}\psi(h^{-2})dh}{\int_{V(s)}^1\psi(h^{-2})dh}\right)\ge c_4,
    \end{align*}
    since $s\le 1/2$. The case $t<s$ follows by reciprocity.
    \end{proof}

    \subsection{Proof of Theorem \ref{t:bndry-G(U)}}\label{ap:s:G(U) approx}
    The proof of Theorem \ref{t:bndry-G(U)} follows the line of  the detailed proof of \cite[Theorem 3.6]{Bio23cpaa},  to which we will often refer, see \cite[Appendix A.2]{Bio23cpaa}.
    
    Let $\epsilon=\epsilon(D) >0$ be such that the map $\Phi:\partial D\times (-\epsilon, \epsilon)\to \R^d$ defined by $\Phi(y,\delta)=y+\delta\mathbf{n}(y)$ defines a diffeomorphism to its image, cf.~\cite[Remark 3.1]{AGCV19}. Here $\mathbf{n}$ denotes the unit interior normal.  Without loss of generality assume that $\epsilon<\mathrm{diam}(D)/20$.

	\begin{lem}\label{apx:l:GDFI close-to-bdry}
		Let $\eta<\epsilon$ and assume that  conditions \ref{U1}-\ref{U4} hold true. Then on the set $\{x\in D:\delta_D(x)<\eta/2\}$ we have 
		\begin{align}\label{eq:close-to-bdry}
			\begin{split}
				\GDP\big(U(\de)\1_{\{\de<\eta\}}\big)&\asymp \frac{1}{\de V(\de)\psi(V(\de)^{-2})}\int_0^{\de}U(t)V(t)\, dt\\
                &+ V(\de)\int_{3\de/2}^{\eta}\frac{U(t)}{tV(t)\psi(V(t)^{-2})}\, dt\, + V(\de)\int_0^{\eta}U(t)V(t)\, dt. \\
			\end{split}
		\end{align}
		In particular, $\GDP\big(U(\delta_D)\1_{(\delta_D<\eta)}\big)<\infty$ on $\{x\in D:\delta_D(x)<\eta/2\}$ if and only if  the integrability condition \ref{U1} holds true. Moreover, all comparability constants depend only on $d$, $D$, $\phi$ and $\psi$ and are independent of $\eta$.
	\end{lem}
    \begin{proof}
    	We repeat the main steps of the procedure done in \cite[Lemma A.3]{Bio23cpaa}. Fix some  $r_0<\epsilon$ and fix $x\in D$ as in the statement. Define
        \begin{align*}
        	\begin{split}
        		D_1&= B(x, \delta_D(x)/2)\\
        		D_2&=\{y: \delta_D(y)<\eta\}\setminus B(x, r_0)\\
        		D_3&=\{y: \delta_D(y)<\delta_D(x)/2\}\cap B(x, r_0)\\
        		D_4&= \{y:3\delta_D(x)/2 < \delta_D(y)<\eta\}\cap B(x,r_0)\\
        		D_5&=\{y: \delta_D(x)/2<\delta_D(y)<3\delta_D(x)/2\}\cap (B(x, r_0)\setminus B(x, \delta_D(x)/2)).
        	\end{split}
        \end{align*}
        Thus, we have that
        $$
        \GDP\big(U(\de)\1_{\{\de<\eta\}}\big)(x)=\sum_{j=1}^5 \int_{D_j}\GDP(x,y)U(\de(y))\, dy =:\sum_{j=1}^5 I_j.
        $$ 
        
        \noindent
        {\bf Estimate of $I_1$:} By repeating the steps as in \cite[Estimate of $I_1$ in Lemma A.3]{Bio23cpaa}, we get
        \begin{equation*}
        	I_1\asymp \frac{U(\de(x))}{\psi(V(\de(x))^{-2})}\lesssim \frac{1}{\de(x)V(\de(x))\psi(V(\de(x))^{-2})}\int_0^{\de(x)}U(t)V(t)\,dt.
        \end{equation*}
        
        \noindent
        {\bf Estimate of $I_2$:} By imitating \cite[Estimate of $I_2$ in Lemma A.3]{Bio23cpaa}, we get
        \begin{equation*}
        	I_2\asymp V(\de(x)) \int_0^{\eta}U(t)V(t)\, dt\, .
        \end{equation*}
        
        For $I_3-I_5$, in addition to \cite[Eq. (A.9)]{Bio23cpaa}, we also need
        \begin{align*}
        	\int_0^a\frac{s^{d-2}}{(1+s)^{d+4}}\asymp 1,\quad a\ge 1,
        \end{align*}
        since the integral can be calculated explicitly.
        
        \noindent 
        {\bf Estimate of $I_3$:} By following the procedure in \cite[Estimate of $I_3$ in Lemma A.3]{Bio23cpaa}, we can prove that  
        \begin{equation*}
        	I_3\, \asymp\,\frac{1}{\de(x)V(\de(x))\psi(V(\de(x))^{-2})} \int_0^{\de(x)}U(t)V(t)\, dt.
        \end{equation*}
        
        \noindent
        {\bf Estimate of $I_4$:} By applying the similar calculations as in $I_3$, see \cite[Estimate of $I_4$ in Lemma A.3]{Bio23cpaa}, it holds that
        \begin{equation*}
        	I_4\,\asymp\, V(\de(x))\int_{3\de(x)/2}^{\eta} \frac{U(t)}{tV(t)\psi(V(t)^{-2})}\, dt\,.
        \end{equation*}
        
        	\noindent
        {\bf Estimate of $I_5$:}  As in \cite[Estimate of $I_5$ in Lemma A.3]{Bio23cpaa}, we can see that
        \begin{equation*}
        	I_5 \,\lesssim\, \frac{U(\de(x))}{\psi(V(\de(x))^{-2})}\,\lesssim\, \frac{1}{\de(x)V(\de(x))\psi(V(\de(x))^{-2})}\int_0^{\de(x)}U(t)V(t)\, dt\, .
        \end{equation*}
        
        Since $I_1 + I_3\lesssim I_3$, the proof is finished.

    \end{proof}

        \begin{lem}\label{apx:l:away-from-bdry}
		Let $\eta<\epsilon$ and assume that  conditions \ref{U1}-\ref{U4} hold true. There exists $c(\eta)>0$ such that for any $x\in D$ satisfying $\de(x)\ge\eta/2$, 
		\begin{equation*}
			\GDP\big(U(\de)\1_{\{\de<\eta\}}\big)(x)\le c(\eta)\, .
		\end{equation*}
	\end{lem}

    \begin{proof}
        Here we employ the same strategy as in \cite[Lemma A.4]{Bio23cpaa}, so we repeat only the main steps. Take $x\in D$ and split the set $\{y:\de(y)< \eta\}$ into 
        \begin{align*}
        	D_1=\{y:\, \de(y)<\eta/4\},\qquad
        	D_2=\{y:\, \eta/4 \le \de(y) <\eta\},
        \end{align*}
        so
        $$
        \GDP\big(U(\delta_D)\1_{(\delta_D<\eta)}\big)(x)=\sum_{j=1}^2 \int_{D_j}\GDP(x,y)U(\delta_D(y))\, dy =:\sum_{j=1}^2 J_j.
        $$
        {\bf Estimate of $J_1$:} 
        By repeating the steps from  \cite[Estimate of $J_1$ in Lemma A.4]{Bio23cpaa}, we get
        \begin{equation*}
        	J_1\lesssim\frac{1}{V(\eta)^2\psi(V(\eta)^{-2})} \int_0^{\eta}U(t)V(t)\, dt.
        \end{equation*}
        
        	\noindent
        {\bf Estimate of $J_2$:}
       	The same calculations as in \cite[Estimate of $J_2$ in Lemma A.4]{Bio23cpaa} yield
        \begin{equation*}
        	J_2\lesssim U(\eta/4).
        \end{equation*}
        
        The estimates $J_1$ and $J_2$ substantiate the claim.       
    \end{proof}

    \begin{proof}[{Proof of Theorem \ref{t:bndry-G(U)}}]
    The proof follows the lines of the proof of \cite[Theorem 3.6]{Bio23cpaa}. Fix some $\eta<\epsilon$. On $\{\de(y)\ge \eta\}$, $U$ is bounded by the assumption \ref{U4}, so
		\begin{equation}\label{eq:gpe-b}
			\GDP \big(U(\de)\1_{\{\de\ge \eta\}}\big)(x)\lesssim\GDP V(\de)(x)\asymp V(\de(x))\,,\quad x\in D,
		\end{equation}
		where the last step estimate follows from Lemma \ref{l:Green invariant}. On $\{\de(x)\ge\eta/2\}$ we have
		\begin{align*}
			\GDP \big(U(\de)\1_{\{\de\ge \eta\}}\big)(x)&\gtrsim\int_{B(x,\eta/4)}\frac{1}{|x-y|^{d}\psi(V(|x-y|)^{-2})}dy\asymp \frac{1}{\psi(V(\eta)^{-2})}\gtrsim 1.
		\end{align*}
		On $\{\de(x)\le\eta/2\}$, since $|x-y|\ge \eta/2$, it holds that
		\begin{align*}
			\GDP \big(U(\de)\1_{\{\de\ge \eta\}}\big)(x)&\gtrsim \int_{\{\de(y)\ge \eta\}}\frac{V(\de(x))V(\de(y))}{|x-y|^{d}V(|x-y|)^2\psi(V(|x-y|)^{-2})}dy\gtrsim V(\de(x)).
		\end{align*}
		Note that, $\de(x)\asymp 1$ on $\{\de(x)\ge\eta/2\}$ which yields $\GDP \big(U(\de)\1_{\{\de\ge \eta\}}\big)(x)\asymp V(\de(x))$ in $D$. In the case $\de(x)\ge \eta/2$, by Lemma \ref{apx:l:away-from-bdry}, we get 
		that $\GDP(U(\de))(x)\asymp 1$, and the right-hand side of \eqref{eq:G(U) sharp bound} is also comparable to 1.
		
        Let us deal now with $\de(x)<\eta/2$. Since we treat $\eta$ as a constant, we can see that $\eta$ in \eqref{eq:close-to-bdry} can be replaced by $\diam D$. The claim now follows from Lemma \ref{apx:l:GDFI close-to-bdry} and \eqref{eq:gpe-b}.
	\end{proof}

\subsection*{Acknowledgment}
This research was partly supported by the Croatian Science Foundation under the project IP-2022-10-2277.

The first-named author also acknowledges financial support under the National Recovery and Resilience Plan (NRRP), Mission 4, Component 2, Investment 1.1, Call for tender No. 104 published on 2.2.2022 by the Italian Ministry of University and Research (MUR), funded by the European Union – NextGenerationEU– Project Title “Non–Markovian Dynamics and Non-local Equations” – 202277N5H9 - CUP: D53D23005670006 - Grant Assignment Decree No. 973 adopted on June 30, 2023, by the Italian Ministry of University and Research (MUR).

	\bibliographystyle{abbrv}
	\bibliography{bibliography}

\begin{thebibliography}{10}

\bibitem{Aba15a}
N.~Abatangelo.
\newblock Large {$s$}-harmonic functions and boundary blow-up solutions for the
  fractional {L}aplacian.
\newblock {\em Discrete and Continuous Dynamical Systems},
  \textbf{35}(12):5555--5607, 2015.

\bibitem{Aba17}
N.~Abatangelo.
\newblock Very large solutions for the fractional {L}aplacian: towards a
  fractional {K}eller-{O}sserman condition.
\newblock {\em Advances in Nonlinear Analysis}, \textbf{6}(4):383--405, 2017.

\bibitem{AbaDup-Nonhomog2017}
N.~Abatangelo and L.~Dupaigne.
\newblock Nonhomogeneous boundary conditions for the spectral fractional
  {L}aplacian.
\newblock {\em Annales de l'Institut Henri Poincar\'{e} C. Analyse Non
  Lin\'{e}aire}, 34(2):439--467, 2017.

\bibitem{AGCV19}
N.~Abatangelo, D.~G\'omez-Castro, and J.~L. V\'azquez.
\newblock Singular boundary behaviour and large solutions for fractional
  elliptic equations.
\newblock {\em Journal of the London Mathematical Society. Second Series},
  \textbf{107}(2):568--615, 2023.

\bibitem{BaeumerLuksMeersc-SpaceTime18}
B.~Baeumer, T.~Luks, and M.~M. Meerschaert.
\newblock Space-time fractional {D}irichlet problems.
\newblock {\em Mathematische Nachrichten}, \textbf{291}(17-18):2516--2535,
  2018.

\bibitem{BC17}
M.~Ben~Chrouda.
\newblock Existence and nonexistence of positive solutions to the fractional
  equation {$\Delta^{\frac\alpha2}u=-u^\gamma$} in bounded domains.
\newblock {\em Annales Academi\ae Scientiarum Fennic\ae . Mathematica},
  \textbf{42}(2):997--1007, 2017.

\bibitem{BCBF16}
M.~Ben~Chrouda and M.~Ben~Fredj.
\newblock Blow up boundary solutions of some semilinear fractional equations in
  the unit ball.
\newblock {\em Nonlinear Analysis. Theory, Methods \& Applications},
  \textbf{140}:236--253, 2016.

\bibitem{BCBF18}
M.~Ben~Chrouda and M.~Ben~Fredj.
\newblock Nonnegative entire bounded solutions to some semilinear equations
  involving the fractional {L}aplacian.
\newblock {\em Potential Analysis}, \textbf{48}(4):495--513, 2018.

\bibitem{Bio21jmaa}
I.~Bio\v{c}i\'{c}.
\newblock Representation of harmonic functions with respect to subordinate
  {B}rownian motion.
\newblock {\em Journal of Mathematical Analysis and Applications},
  \textbf{506}(1):Paper No. 125554, 31, 2022.

\bibitem{Bio23cpaa}
I.~Bio\v{c}i\'{c}.
\newblock Semilinear {D}irichlet problem for subordinate spectral {L}aplacian.
\newblock {\em Communications on Pure and Applied An alysis},
  \textbf{22}(3):851--898, 2023.

\bibitem{BVW-na21}
I.~Bio\v{c}i\'{c}, Z.~Vondra\v{c}ek, and V.~Wagner.
\newblock Semilinear equations for non-local operators: beyond the fractional
  {L}aplacian.
\newblock {\em Nonlinear Analysis. Theory, Methods \& Applications},
  \textbf{207}:Paper No. 112303, 40, 2021.

\bibitem{BW25}
I.~Bio\v{c}i\'c and V.~Wagner.
\newblock Large solutions for subordinate spectral {L}aplacian.
\newblock {\em Nonlinear Differential Equations and Applications},
  \textbf{33}(1):Paper No. 9, 46, 2026.

\bibitem{BJ20}
A.~Biswas and S.~Jarohs.
\newblock On overdetermined problems for a general class of nonlocal operators.
\newblock {\em Journal of Differential Equations}, \textbf{268}(5):2368--2393,
  2020.

\bibitem{BL-na21}
A.~Biswas and J.~L\H{o}rinczi.
\newblock Hopf's lemma for viscosity solutions to a class of non-local
  equations with applications.
\newblock {\em Nonlinear Analysis. Theory, Methods \& Applications},
  \textbf{204}:Paper No. 112194, 18, 2021.

\bibitem{BMS-boundaryreg-23}
A.~Biswas, M.~Modasiya, and A.~Sen.
\newblock Boundary regularity of mixed local-nonlocal operators and its
  application.
\newblock {\em Annali di Matematica Pura ed Applicata. Series IV},
  \textbf{202}(2):679--710, 2023.

\bibitem{bliedtner}
J.~Bliedtner and W.~Hansen.
\newblock {\em Potential theory}.
\newblock Springer-Verlag, 1986.

\bibitem{BogBycz-PoteThe1999}
K.~Bogdan and T.~Byczkowski.
\newblock Potential theory for the {$\alpha$}-stable {S}chr\"odinger operator
  on bounded {L}ipschitz domains.
\newblock {\em Studia Mathematica}, \textbf{133}(1):53--92, 1999.

\bibitem{bogdan_density_and_tails_unimodal}
K.~Bogdan, T.~Grzywny, and M.~Ryznar.
\newblock Density and tails of unimodal convolution semigroups.
\newblock {\em Journal of Functional Analysis}, \textbf{266}(6):3543--3571,
  2014.

\bibitem{barriers_BGR}
K.~Bogdan, T.~Grzywny, and M.~Ryznar.
\newblock Barriers, exit time and survival probability for unimodal {L}\'{e}vy
  processes.
\newblock {\em Probability Theory and Related Fields},
  \textbf{162}(1-2):155--198, 2015.

\bibitem{bogdan_et_al_19}
K.~Bogdan, S.~Jarohs, and E.~Kania.
\newblock Semilinear {Dirichlet} problem for the fractional {Laplacian}.
\newblock {\em Nonlinear Analysis. Theory, Methods \& Applications},
  \textbf{193}:111512, 2020.

\bibitem{BFV-cvPDE18}
M.~Bonforte, A.~Figalli, and J.~L. V\'azquez.
\newblock Sharp boundary behaviour of solutions to semilinear nonlocal elliptic
  equations.
\newblock {\em Calculus of Variations and Partial Differential Equations},
  \textbf{57}(2):Paper No. 57, 34, 2018.

\bibitem{CaffarelliKato}
L.~A. Caffarelli and Y.~Sire.
\newblock On some pointwise inequalities involving nonlocal operators.
\newblock In {\em Harmonic analysis, partial differential equations and
  applications}, Appl. Numer. Harmon. Anal., pages 1--18.
  Birkh\"auser/Springer, Cham, 2017.

\bibitem{CGcV-jfa21}
H.~Chan, D.~G\'omez-Castro, and J.~L. V\'azquez.
\newblock Blow-up phenomena in nonlocal eigenvalue problems: when theories of
  {$L^1$} and {$L ^2$} meet.
\newblock {\em Journal of Functional Analysis}, \textbf{280}(7):Paper No.
  108845, 68, 2021.

\bibitem{CFQ}
H.~Chen, P.~Felmer, and A.~Quaas.
\newblock Large solutions to elliptic equations involving fractional
  {L}aplacian.
\newblock {\em Annales de l'Institut Henri Poincar\'{e} C. Analyse Non
  Lin\'{e}aire}, \textbf{32}(6):1199--1228, 2015.

\bibitem{chen_veron}
H.~Chen and L.~V\'{e}ron.
\newblock Semilinear fractional elliptic equations involving measures.
\newblock {\em Journal of Differential Equations}, \textbf{257}(5):1457--1486,
  2014.

\bibitem{ChenSong-Two-sided-JFA05}
Z.-Q. Chen and R.~Song.
\newblock Two-sided eigenvalue estimates for subordinate processes in domains.
\newblock {\em Journal of Functional Analysis}, \textbf{226}(1):90--113, 2005.

\bibitem{ChenSong-Spectral07}
Z.-Q. Chen and R.~Song.
\newblock Spectral properties of subordinate processes in domains.
\newblock In {\em Stochastic analysis and partial differential equations},
  volume \textbf{429} of {\em Contemp. Math.}, pages 77--84. Amer. Math. Soc.,
  Providence, RI, 2007.

\bibitem{CKSV22}
S.~Cho, P.~Kim, R.~Song, and Z.~Vondraček.
\newblock Heat kernel estimates for subordinate {M}arkov processes and their
  applications.
\newblock {\em Journal of Differential Equations}, \textbf{316}:28--93, 2022.

\bibitem{CVW25}
I.~Chowdhury, Z.~Vondra{\v{c}}ek, and V.~Wagner.
\newblock Large solutions to semilinear equations for subordinate {L}aplacians
  in $ {C}^{1,1}$ bounded open sets.
\newblock {\em Annali di Matematica}, 2025.

\bibitem{chung_zhao}
K.~L. Chung and Z.~X. Zhao.
\newblock {\em {F}rom {B}rownian motion to {S}chr\"{o}dinger's equation},
  volume \textbf{312}.
\newblock Springer-Verlag, 2001.

\bibitem{dhifli2012}
A.~Dhifli, H.~M\^{a}agli, and M.~Zribi.
\newblock On the subordinate killed {B}.{M} in bounded domains and existence
  results for nonlinear fractional {D}irichlet problems.
\newblock {\em Mathematische Annalen}, \textbf{352}(2):259--291, 2012.

\bibitem{DGcV-Kato}
J.~I. D\'iaz, D.~G\'omez-Castro, and J.~L. V\'azquez.
\newblock The fractional {S}chr\"odinger equation with general nonnegative
  potentials. {T}he weighted space approach.
\newblock {\em Nonlinear Analysis. Theory, Methods \& Applications}, \textbf{
  177}:325--360, 2018.

\bibitem{F20}
M.~M. Fall.
\newblock Semilinear elliptic equations for the fractional {L}aplacian with
  {H}ardy potential.
\newblock {\em Nonlinear Analysis. Theory, Methods \& Applications},
  \textbf{193}:111311, 29, 2020.

\bibitem{FQ12}
P.~Felmer and A.~Quaas.
\newblock Boundary blow up solutions for fractional elliptic equations.
\newblock {\em Asymptotic Analysis}, \textbf{78}(3):123--144, 2012.

\bibitem{Fernandez-RealRos-Oton2014}
X.~Fern\'{a}ndez-Real and X.~Ros-Oton.
\newblock Boundary regularity for the fractional heat equation.
\newblock {\em Revista de la Real Academia de Ciencias Exactas, F\'{\i}sicas y
  Naturales. Serie A. Matematicas. RACSAM}, \textbf{110}(1):49--64, 2016.

\bibitem{FJ23}
P.~A. Feulefack and S.~Jarohs.
\newblock Nonlocal operators of small order.
\newblock {\em Annali di Matematica Pura ed Applicata},
  \textbf{202}(4):1501--1529, 2023.

\bibitem{Figalli-DeGiorgi}
A.~Figalli and J.~Serra.
\newblock On stable solutions for boundary reactions: a {D}e {G}iorgi-type
  result in dimension {$4+1$}.
\newblock {\em Inventiones Mathematicae}, \textbf{ 219}(1):153--177, 2020.

\bibitem{FOT}
M.~Fukushima, Y.~Oshima, and M.~Takeda.
\newblock {\em Dirichlet forms and symmetric {M}arkov processes}.
\newblock Walter de Gruyter, 2011.

\bibitem{grubb2016}
G.~Grubb.
\newblock Regularity of spectral fractional {D}irichlet and {N}eumann problems.
\newblock {\em Mathematische Nachrichten}, \textbf{289}(7):831--844, 2016.

\bibitem{BogdanHansen}
W.~Hansen and K.~Bogdan.
\newblock Positive harmonically bounded solutions for semi-linear equations.
\newblock {\em Journal of Differential Equations}, \textbf{443}:113544, 2025.

\bibitem{HuynhNguyen2022}
P.-T. Huynh and P.-T. Nguyen.
\newblock Semilinear nonlocal elliptic equations with source term and measure
  data.
\newblock {\em Journal d'Analyse Math\'ematique}, \textbf{149}(1):49--111,
  2023.

\bibitem{HuynhNguyen2022_new}
P.-T. Huynh and P.-T. Nguyen.
\newblock Compactness of green operators with applications to semilinear
  nonlocal elliptic equations.
\newblock {\em Journal of Differential Equations}, \textbf{418}:97--141, 2025.

\bibitem{ikeda-watanabe}
N.~Ikeda and S.~Watanabe.
\newblock On some relations between the harmonic measure and the {L}\'evy
  measure for a certain class of {M}arkov processes.
\newblock {\em Journal of Mathematics of Kyoto University}, \textbf{2}:79--95,
  1962.

\bibitem{KKLL-jfa19}
M.~Kim, P.~Kim, J.~Lee, and K.-A. Lee.
\newblock Boundary regularity for nonlocal operators with kernels of variable
  orders.
\newblock {\em Journal of Functional Analysis}, \textbf{277}(1):279--332, 2019.

\bibitem{KM-ejp18-heatkerneldomain}
P.~Kim and A.~Mimica.
\newblock Estimates of {D}irichlet heat kernels for subordinate {B}rownian
  motions.
\newblock {\em Electronic Journal of Probability}, \textbf{23}:Paper No. 64,
  45, 2018.

\bibitem{KSV_heat-PA18}
P.~Kim, R.~Song, and Z.~Vondra{\v{c}}ek.
\newblock Heat kernels of non-symmetric jump processes: beyond the stable case.
\newblock {\em Potential Analysis}, \textbf{49}(1):37--90, 2018.

\bibitem{ksv2020boundary}
P.~Kim, R.~Song, and Z.~Vondra{\v{c}}ek.
\newblock On the boundary theory of subordinate killed {L}{\'e}vy processes.
\newblock {\em Potential Analysis}, \textbf{53}(1):131--181, 2020.

\bibitem{KSV_bhpinf-PA14}
P.~Kim, R.~Song, and Z.~Vondra\v{c}ek.
\newblock Boundary {H}arnack principle and {M}artin boundary at infinity for
  subordinate {B}rownian motions.
\newblock {\em Potential Analysis}, \textbf{41}(2):407--441, 2014.

\bibitem{KSV14-gubhp-spa}
P.~Kim, R.~Song, and Z.~Vondra\v{c}ek.
\newblock Global uniform boundary {H}arnack principle with explicit decay rate
  and its application.
\newblock {\em Stochastic Processes and their Applications},
  \textbf{124}(1):235--267, 2014.

\bibitem{ksv_minimal2016}
P.~Kim, R.~Song, and Z.~Vondra\v{c}ek.
\newblock Minimal thinness with respect to subordinate killed {B}rownian
  motions.
\newblock {\em Stochastic Processes and their Applications},
  \textbf{126}(4):1226--1263, 2016.

\bibitem{ksv_RMI18}
P.~Kim, R.~Song, and Z.~Vondra\v{c}ek.
\newblock Accessibility, {M}artin boundary and minimal thinness for {F}eller
  processes in metric measure spaces.
\newblock {\em Revista Matem\'atica Iberoamericana}, \textbf{34}(2):541--592,
  2018.

\bibitem{KlimsiakRozkosz25}
T.~Klimsiak and A.~Rozkosz.
\newblock Dirichlet problem for semilinear partial integro-differential
  equations: the method of orthogonal projection.
\newblock {\em Calculus of Variations and Partial Differential Equations},
  \textbf{64}(1):Paper No. 2, 39, 2025.

\bibitem{KW_martin}
H.~Kunita and T.~Watanabe.
\newblock Markov processes and {M}artin boundaries. {I}.
\newblock {\em Illinois Journal of Mathematics}, \textbf{9}:485--526, 1965.

\bibitem{Mimica-heatSBM-proceed16}
A.~Mimica.
\newblock Heat kernel estimates for subordinate {B}rownian motions.
\newblock {\em Proceedings of the London Mathematical Society. Third Series},
  \textbf{113}(5):627--648, 2016.

\bibitem{Okura}
H.~{\^ O}kura.
\newblock Recurrence and transience criteria for subordinated symmetric
  {M}arkov processes.
\newblock {\em Forum Mathematicum}, \textbf{14}:121--146, 2002.

\bibitem{Ros-OtonSerra2016}
X.~Ros-Oton and J.~Serra.
\newblock Regularity theory for general stable operators.
\newblock {\em Journal of Differential Equations}, \textbf{260}(12):8675--8715,
  2016.

\bibitem{sato_PO}
K.~Sato.
\newblock Potential operators for {M}arkov processes.
\newblock In {\em Proceedings of the {S}ixth {B}erkeley {S}ymposium on
  {M}athematical {S}tatistics and {P}robability, {V}ol. {III}: {P}robability
  theory}, pages 193--211, 1972.

\bibitem{sato}
K.~Sato.
\newblock {\em {L}\'{e}vy processes and infinitely divisible distributions},
  volume \textbf{68}.
\newblock Cambridge University Press, 1999.

\bibitem{bernstein}
R.~L. Schilling, R.~Song, and Z.~Vondra\v{c}ek.
\newblock {\em Bernstein functions: theory and applications}.
\newblock De Gruyter, 2012.

\bibitem{SV2014}
R.~Servadei and E.~Valdinoci.
\newblock On the spectrum of two different fractional operators.
\newblock {\em Proceedings of the Royal Society of Edinburgh. Section A.
  Mathematics}, \textbf{144}(4):831--855, 2014.

\bibitem{SongVondracek2006potenSpecial}
R.~Song and Z.~Vondra\v{c}ek.
\newblock Potential theory of special subordinators and subordinate killed
  stable processes.
\newblock {\em Journal of Theoretical Probability}, \textbf{19}(4):817--847,
  2006.

\bibitem{yosida_FA}
K.~Yosida.
\newblock {\em Functional analysis}.
\newblock Springer-Verlag, 1995.

\end{thebibliography}

	\bigskip
	
	\noindent{\bf Ivan Bio\v{c}i\'c}
	
	\noindent Department of Mathematics, Faculty of Science, University of Zagreb, Zagreb, Croatia,
	
	\noindent Email: \texttt{ibiocic@math.hr}
    
        \noindent Department of Mathematics ``Giuseppe Peano", University of Turin, Turin, Italy,
	
	\noindent Email: \texttt{ivan.biocic@unito.it}
	
	\bigskip
	
	\noindent{\bf Vanja Wagner}
	
	\noindent Department of Mathematics, Faculty of Science, University of Zagreb, Zagreb, Croatia,
	
	\noindent Email: \texttt{wagner@math.hr}

\end{document}